\documentclass[leqno,english]{amsart}
\usepackage{amsmath,amssymb,amsthm,enumerate,esint,mathtools,xparse}
\usepackage[svgnames]{xcolor}
  \usepackage{hyperref}
 \usepackage{empheq}

\usepackage{bbm}\DeclareMathOperator{\csch}{csch}

 \usepackage{mathrsfs} 
\usepackage{mathabx}

\numberwithin{equation}{section}
\newtheorem{definition}{Definition}[section]
\newtheorem{theorem}{Theorem}[section]

\newtheorem{lemma}[theorem]{Lemma}
\newtheorem{proposition}[theorem]{Proposition}

\newtheorem{remark}[theorem]{Remark}

\newcommand{\bke}[1]{\left ( #1 \right )}
\newcommand{\bkt}[1]{\left [ #1 \right ]}
\newcommand{\bket}[1]{\left \{ #1 \right \}}
\newcommand{\norm}[1]{\left \| #1 \right \|}

\newcommand{\abs}[1]{\left | #1 \right |}

\newcommand\al{\alpha}
\newcommand\be{\beta}
\newcommand\ga{\gamma}
\newcommand\de{\delta}

\newcommand\ve{\varepsilon}
\newcommand\e {\varepsilon}

\renewcommand\th{\theta}
\newcommand\ka{\kappa}
\newcommand\la{\lambda}
\newcommand\si{\sigma}

\newcommand\om{\omega}
\newcommand\Ga{\Gamma}
\newcommand\De{\Delta}

\newcommand\La{\Lambda}

\newcommand\Rg{\mathscr{R}_g}
\newcommand{\R}{\mathbb{R}}
\newcommand{\CC}{\mathbb{C}}

\renewcommand{\Re} {\mathop{\mathrm{Re}}\nolimits}
\renewcommand{\Im} {\mathop{\mathrm{Im}}\nolimits}

\newcommand{\pd}{\partial}
\newcommand{\nb}{\nabla}

\newcommand{\wt}[1]{\widetilde {#1}}

\renewcommand{\bar}[1]{\overline{#1}}

\newcommand{\EQ}[1]{\begin{equation}\begin{split} #1 \end{split}\end{equation}}
\newcommand{\EQN}[1]{\begin{equation*}\begin{split} #1 \end{split}\end{equation*}}

\DeclarePairedDelimiter{\oldnormaux}{\bracevert}{\bracevert}

\NewDocumentCommand{\oldnorm}{som}{%
  \IfBooleanTF{#1}
    {\oldnormaux*{#3}}
    {\IfNoValueTF{#2}
       {\oldnormaux*{\vphantom{dq}#3}}
       {\oldnormaux[#2]{#3}}%
    }%
}

\usepackage{scalerel}

\begin{document}
\title[Relaxation of oscillating bubbles]{Thermal relaxation toward equilibrium and periodically pulsating spherical bubbles 
in an incompressible liquid }

\author[C. Lai]{Chen-Chih Lai}
\address{\noindent
Department of Mathematics,
Columbia University , New York, NY, 10027, USA}
\email{cl4205@columbia.edu}

\author[M. I. Weinstein]{Michael I. Weinstein}
\address{\noindent
Department of Applied Physics and Applied Mathematics and Department of Mathematics,
Columbia University , New York, NY, 10027, USA}
\email{miw2103@columbia.edu}

\begin{abstract}
We study the radial relaxation dynamics toward equilibrium and time-periodic pulsating spherically symmetric gas bubbles in an incompressible liquid due to thermal effects.  The asymptotic model \cite{Prosperetti-JFM1991,BV-SIMA2000} is one where the pressure within the gas bubble is spatially uniform
 and satisfies an ideal gas law, relating the pressure, density and temperature of the gas.  The temperature of the surrounding liquid is taken to be constant and the behavior of the liquid pressure at infinity is prescribed to be constant or periodic in time.
In \cite{LW-vbas2022}, for the case where the liquid pressure at infinity  is a positive constant, we proved the existence of a one-parameter manifold of spherical equilibria, parameterized by the bubble mass, and further proved that it  is  a nonlinearly and exponentially asymptotically stable center manifold.

In the present article, we first 
refine the exponential time-decay estimates, via a study of the linearized dynamics subject to the constraint of fixed mass. We obtain, in particular, estimates for the exponential decay rate constant, which highlight the interplay between the effects of thermal diffusivity and the liquid viscosity. 

We then study the nonlinear radial dynamics of the bubble-fluid system subject to a pressure field at infinity which is a small-amplitude and time-periodic perturbation about a positive constant. We prove that nonlinearly and exponentially asymptotically stable time-periodically pulsating solutions 
of the nonlinear (asymptotic) model exist for all sufficiently small forcing amplitudes. The existence of such states is formulated as a fixed point problem for the Poincar\'e return map, and the existence of a fixed point makes use of our (constant mass constrained) exponential time-decay estimates of the linearized problem.  
\end{abstract}
\maketitle





\tableofcontents

\section{Introduction; the model and main results}

Consider a spherical gas bubble, immersed in a liquid. When perturbed from its equilibrium radius 
(for which the pressure inside the bubble is balanced with the pressure at infinity), the bubble will undergo expansion and contraction. There are different mechanisms that can contribute to the damping of these oscillations. Most important are viscous forces at the gas-liquid interface, acoustic radiation of sound waves toward infinity if the surrounding liquid is compressible,
 and  thermal conduction between the gas in the bubble and the surrounding liquid. Other effects which arise are 
 mass diffusion and the evaporation-condensation of vapor  \cite{Prosperetti-ApplSciRes1982}, \cite{Fanelli-Prosperetti-Reali-Acustica(47)1981}, \cite{Fanelli-Prosperetti-Reali-Acustica(49)1981}, 
 in which mass is transferred across the bubble - fluid interface.
 In this paper we focus on the relaxation dynamics of a gas bubble, immersed in an incompressible fluid,
 due to thermal effects. In the class of models we consider, the mass of the gas bubble remains constant throughout the evolution; there is no mass-transfer across the bubble - fluid interface.


The dynamics of a gas bubble in a liquid depends strongly on the variation of  pressure of the  gas inside the bubble.
When thermal effects are neglected, the gas can be assumed to satisfy a polytropic law, in which
the gas pressure, $p_g$,  is proportional to a reciprocal power of the bubble volume: $p_g= p_*(R_*/R)^{3n}$. Here $R$ denotes the bubble radius, $n$ is the polytropic constant, and $p_*$ and $R_*$ denote  equilibrium values. 
Such models have been studied extensively in  \cite{Lauterborn-JAcoustSocAm1976, Prosperetti-JAcoustSocAm1974, Prosperetti-JAcoustSocAm1975, Keller-Miksis-JAcoustSocAm1980, Lastman-Wentzell-JAcoustSocAm1981, Lastman-Wentzell-JAcoustSocAm1982,SW-SIMA2011}. For reviews of the subject, see e.g. \cite{Plesset-Prosperetti-1977, Prosperetti-Ultrasonics1984a, Prosperetti-Ultrasonics1984b, Apfel-Ultrasonics1981, Leighton2011review}.
Corrections to the polytropic assumption, to account for thermal effects, were explored in, for example, in
\cite{Flynn-JAcoustSocAm1975,Prosperetti-JAcoustSocAm1977,Prosperetti-Ultrasonics1984a}.
%
 %
 
A more accurate model was proposed by Prosperetti in \cite{Prosperetti-JFM1991}. This model incorporates the dissipative effects of thermal diffusion, and of liquid viscosity at the bubble-fluid interface.  In the asymptotic regime considered,  the pressure within the gas bubble is taken to be spatially uniform
 and to satisfy the ideal gas law, relating pressure, density and temperature: 
  $p_g(t) = \Rg T_g(r,t) \rho_g(r,t)$. Here, $\Rg$ is the specific gas constant, and $T_g$ and $\rho_g$ are the temperature and the density of the gas. The temperature of the liquid surrounding  the bubble is taken to be constant, and the liquid pressure at infinity is taken to be a positive constant or a small amplitude time-periodic pertrubation of a spatially uniform positive constant pressure.
  The assumption that the pressure within the gas is uniform is justified in  \cite[\S6]{Prosperetti-JFM1991}  and \cite[Appendix A]{BV-SIMA2000}.  In \cite{Prosperetti-JFM1991}, the radial relaxation dynamics in this model were studied  for time-periodic and spatially uniform fluid pressure at infinity, in the context of the linearized equations and other simplifying approximations. 
Thermal dissipation rates were investigated in \cite{Prosperetti-JFM1991} within two regimes: the nearly isothermal regime corresponding to rapid thermal diffusion, and the nearly adiabatic regime corresponding to slow thermal diffusion.
In these two regimes, Prosperetti derived explicit approximate thermal dissipation rates and presented several numerical results.
Over the decades, the model has been further studied via different types of approximations. 
For example, quadratic and biquadratic approximations are used in \cite{ZP-JFM2020} to reduce the full PDE model to a simple ODE model for which the computational cost in simulation is significantly reduced. 
See also \cite[Sec. 3.3.4]{Hegedus-thesis2018} for efficiency comparison for Galerkin polynomial, Galerkin trigonometric, Galerkin hat function, spectral collocation and finite difference methods.
In  \cite{BV-SIMA2000} this model --with zero fluid interfacial  viscosity-- was deduced using non-dimensionalization and order of magnitude comparisons of non-dimensional parameters arising  in  sonoluminescence experiments \cite{barber1991observation, barber1992light, barber1994sensitivity}. The article  \cite{BV-SIMA2000} presents rigorous results on local in time well-posedness  and  Lyapunov stability of spherically symmetric equilibrium bubbles, relative to spherically symmetric perturbations of the same mass.

In \cite{LW-vbas2022}, we gave a rigorous proof that the model of \cite{Prosperetti-JFM1991}, with constant pressure at infinity, has a one-parameter manifold  of spherically symmetric equilibrium bubbles, parameterized by the bubble mass, and that this manifold of equilibria is nonlinearly asymptotically stable with exponential rate of convergence under the nonlinear dynamics; an initially small spherically symmetric perturbation of an equilibrium bubble evolves, as time advances, toward a spherical bubble with the mass of the initial data exponentially fast.

In this paper, we deepen our understanding of the dynamics of bubble oscillations within this model, and extend our study to time-periodic pulsating bubbles induced by external pressure forcing.
 We first establish exponential time-decay of solutions for the linearized evolution, subject to a linear constraint on the initial data, implied by mass-conservation. As part of this analysis, we obtain estimates for the exponential decay rate constant which highlights the interplay between  thermal diffusivity and liquid viscosity.  (In this model, the surrounding liquid is incompressible. Hence, acoustic radiation damping is not present in the model.)
Finally, we study the nonlinear dynamics of the bubble-fluid system forced by a positive, spatially uniform, nearly constant, periodic in time pressure
field at infinity. Using our exponential time-decay estimates of the linearized problem, we prove that the Poincar\'e return map of this periodically forced nonlinear system has a fixed point, which is nonlinearly asymptotically stable. Hence, time-periodic forcing gives rise to nonlinearly exponentially asymptotically stable periodically pulsating bubbles. 
%

In the following subsections of this introduction we provide details of the mathematical model and outline our main results more precisely. 
We then draw comparisons with previous work.

\subsection{Model of a gas bubble in a fluid }

 
Our model of a gas bubble in a liquid, in the presence of thermal effects in the gas and liquid viscosity at the bubble-gas interface, is a free boundary problem of differential equations describing the dynamics of the gas inside the bubble (governed by the compressible Navier-Stokes equations) and the liquid outside the bubble (governed by the incompressible Navier-Stokes equations) and their interaction with the bubble surface \cite[(3.1)-(3.4)]{LW-vbas2022}.
After a spherically symmetric reduction \cite[Section 5]{LW-vbas2022}, the problem is reduced to a system
consisting of  a quasilinear parabolic partial differential equation for the gas density $\rho(r,t)$ coupled to ordinary differential equations which determine the radius of the bubble, $R(t)$, and the constraint that the mass of the bubble is conserved \cite{LW-vbas2022,BV-SIMA2000}.  
As explained above, the gas pressure, $p(t)$, is assumed to be spatially uniform throughout the bubble. The ideal gas law, which relates pressure, temperature, and density, together with the assumption of constant liquid temperature, implies that 
 $p(t)$ is proportional to the density on the bubble surface, $\rho(R(t),t)$.  
 
 For $t>0$, the full mathematical model is
\begin{subequations}
\label{red-eqns}
\begin{align}
\pd_t\rho(r,t) &= 
\frac{\ka}{\ga c_v} \De_r\log\rho(r,t) + \frac{1}{\ga} \frac{\pd_t p(t)}{ p(t)}\Big(  \frac13 r \pd_r\rho(r,t) 
+ \rho(r,t) \Big) ,\quad 0\le r\le R(t), \label{eq-bv-3.10prime}\\
\dot R(t) &= -\frac\ka{\ga c_v} \frac{\pd_r\rho(R(t),t)}{\bke{\rho(R(t),t)}^2} - \frac{R(t)}{3\ga} \frac{\pd_t p(t)}{p(t)},\label{eq-bv-3.15prime}\\
\rho(R(t),t) &= \frac1{\Rg T_\infty} \bkt{p_\infty(t) + \frac{2\si}{R(t)} 
+ 4\mu_l \frac{\dot R}R
+ \rho_l\bke{R(t)\ddot R(t) + \frac32 (\dot R(t))^2} }.\label{eq-bv-3.16prime}\\
&\notag\\
&\hspace{-3.5cm}\text{The gas pressure, $p(t)$, is related to $\rho$ and $R$ via the ideal gas constituitive law: }\notag\\
&\notag\\
p(t) &= \Rg T_\infty \rho(R(t),t),\quad t>0.\label{eq-constitutive-relation}
\end{align}
\end{subequations}
In \eqref{red-eqns}, $\kappa$ denotes the thermal conductivity of the gas and $\mu_l$ the liquid viscosity. The adiabatic   constant, $\ga \equiv 1 + \frac{\Rg}{c_v} = \frac{c_p}{c_v}>1$, is determined by $\Rg$, the specific gas constant, $c_p$, the heat capacity at constant pressure and $c_v$, the heat capacity at constant volume.
The spatially uniform far-field liquid pressure $p_\infty(t)$ is prescribed. 
Equation \eqref{eq-bv-3.16prime} arises from the balance of stresses at the gas-fluid interface.

Equations \eqref{eq-bv-3.10prime}, \eqref{eq-bv-3.15prime} are equivalent to conservation of the bubble mass (no mass-transfer across the bubble-fluid interface). To see this, note that the mass inside the bubble is given by:
\EQ{\label{eq-conserve-mass}
M &= {\rm Mass}[\rho,R] := \int_{B_{R(t)}} \rho(x,t)\, dx,\quad t>0.
}
Differentiation yields: 
\EQN{
\frac{d}{dt} \int_{B_{R(t)}} \rho\, dx 
&= \int_{B_{R(t)}} \pd_t\rho\, dx + \int_{\pd B_{R(t)}} \rho\, \dot{R}\, dS\\
&= \int_{B_{R(t)}} \frac{\ka}{\ga c_v}\, \De_r\log\rho\, dx + \frac1\ga\, \frac{\pd_tp(t)}{p(t)} \int_{B_{R(t)}} \bke{\frac13 r\pd_r\rho + \rho}\, dx + \int_{\pd B_{R(t)}} \rho\, \dot{R}(t)\, dS\\
&= \frac{\ka}{\ga c_v}\, \frac{\pd_r\rho(R(t),t)}{\rho(R(t),t)}\cdot 4\pi R^2(t) + \frac{R(t)}{3\ga}\, \frac{\pd_tp(t)}{p(t)}\, \rho(R(t),t)\cdot 4\pi R^2(t) + \rho(R(t),t)\dot R(t)\cdot4\pi R^2(t)\\
&= 4\pi R^2(t)\rho(R(t),t) \bkt{ \frac{\ka}{\ga c_v}\, \frac{\pd_r\rho(R(t),t)}{(\rho(R(t),t))^2} + \frac{R}{3\ga}\, \frac{\pd_tp(t)}{p(t)} + \dot R(t) },
}
which vanishes  if and only if \eqref{eq-bv-3.15prime} holds.

Hence, our model \eqref{red-eqns} is equivalent to the following family of equations in which 
 the bubble mass parameter,  $M$, is made explicit:
\begin{subequations}
\label{red-eqns-mass}
\begin{align}
\pd_t\rho(r,t) &= 
\frac{\ka}{\ga c_v} \De_r\log\rho(r,t) + \frac{1}{\ga} \frac{\pd_t p(t)}{ p(t)}\Big(  \frac13 r \pd_r\rho(r,t) 
+ \rho(r,t) \Big) ,\quad 0\le r\le R(t), \label{eq-bv-3.10prime-mass}\\
M &= \int_{B_{R(t)}} \rho(x,t)\, dx,\quad t>0,\label{eq-bv-3.15prime-mass}\\
\rho(R(t),t) &= \frac1{\Rg T_\infty} \bkt{p_\infty(t) + \frac{2\si}{R(t)} 
+ 4\mu_l \frac{\dot R}R
+ \rho_l\bke{R(t)\ddot R(t) + \frac32 (\dot R(t))^2} }.\label{eq-bv-3.16prime-mass}\\
p(t) &= \Rg T_\infty \rho(R(t),t),\quad t>0.\label{eq-constitutive-relation-mass}
\end{align}
\end{subequations}
We study the system \eqref{red-eqns-mass} with prescribed initial conditions
\begin{equation}\label{eq:IC}
\quad R(t=0)=R_0, \quad \dot{R}(0)=\dot{R}_0, \quad \textrm{and}\  \rho(r,t=0)=\rho_0(r),\ 0\le r\le R_0.
 \end{equation} 

In \cite{LW-vbas2022}, we studied the equilibria of \eqref{red-eqns},  for a prescribed constant far-field pressure $p_{\infty,*}$. These are  given by the one-parameter family of spherically symmetric bubbles  \eqref{red-eqns}, parametrized by the total mass of the gas bubble, $M$:
\EQ{\label{eq-equilib-manifold}
\mathcal{M}_*&= \bket{(\rho_*[M],R_*[M], \dot R_*=0): 0<M<\infty }.
}
 For each $M>0$,  $(\rho_*[M],R_*[M])$ denotes the unique solution, satisfying $\rho_*>0$ and $R_*>0$, of the pair of the algebraic equations:
\EQ{\label{eq-system-alg}
\begin{cases}
\dfrac{4\pi}3 \rho_* R_*^3 = M, \\
\rho_* = \dfrac1{\Rg T_\infty} \bke{p_{\infty,*} + \dfrac{2\si}{R_*}}.
\end{cases}
}
In \cite[Theorem 6.7]{LW-vbas2022}, we introduced a metric which measures the distance from the state of the system $(\rho(r,t), R(t),\dot{R}(t))$
 to the manifold of equilibria $\mathcal{M}_*$ and proved that
 if $(\rho_0(r), R_0, \dot{R}_0)$ is an initial condition whose metric distance to $\mathcal{M}_*$  is sufficiently small, then the metric distance of   $(\rho(r,t), R(t), \dot{R}(t))$, to $\mathcal{M}_*$ tends to zero exponentially fast as $t\to\infty$. 
 At the heart of our proof are: a proof of nonlinear asymptotic stability with no explicit rate of time-decay
 (based on the energy dissipation law in \cite{BV-SIMA2000}), spectral and time-decay analysis of the corresponding linearized evolution, and a center manifold analysis in which we bootstrap weak time-decay to exponential time-decay of perturbations.
These results extend to the case where the constant far-field pressure $p_{\infty,*}$ replaced by a time dependent far-field pressure $p_\infty(t)$ which approaches $p_{\infty,*}$ sufficiently fast.

\subsection{Summary of the main results} 
{\ }\\ \vspace{-0.33cm}

(I) \emph{Linear asymptotic stability of $\mathcal{M}_*$.} We present two approaches to linear asymptotic stability.
\begin{enumerate}
\item In Theorem \ref{thm-exp-decay-LT} we prove linear asymptotic stability, with an exponential rate of time-decay, of equilibrium solutions $(\rho_*,R_*)$ of \eqref{red-eqns} with constant far-field pressure. 
That is,  under the linearized dynamics, perturbations decay to zero with a rate $\sim \exp(-\beta t)$, where the decay constant, $\beta$, is strictly positive and depends on the parameters in \eqref{red-eqns}.
For such time-decay,  the initial density need only be in $L^2$ and the decay is in the $L^2$ sense.
Decay in H\"older ($C^{2+2\alpha}$) spaces is deduced if the data are assumed to be more regular. The proof  of Theorem \ref{thm-exp-decay-LT} is via Laplace transform and a very detailed analysis centering on  a negative upper bound, $-\beta<0$, for the real parts of the (infinitely many) poles of the resolvent of $\mathcal{L}_-$, the operator which generates the linear time-dynamics.
 
Our  lower bound on the exponential rate, $\beta>0$ is strictly positive for liquid viscosities $\mu_l\ge0$ (and thermal diffusivity $\kappa>0$). In  Appendix \ref{sec-compare-prosperetti} we compare our rigorous lower bound on $\beta$ with the approximations to  $\be$ given in previous results (\cite{Prosperetti-JFM1991}) in different regimes specified by the {\it thermal diffusivity parameter}, $\chi=\kappa/(c_p\rho_*)$: (i) $\chi$ large, the nearly isothermal regime, and (ii) $\chi$ small, the adiabatic regime. 

  \item A second result, of independent interest, on linear asymptotic stability is presented in Theorem \ref{thm-asymp-stab-linear}. 
  It is based on the observation that the full nonlinear dynamics satisfies a  total energy dissipation law.
 Rather than analyze an explicit representation of the solution of the linearized problem, we prove a dissipation law for a linearized energy, as well as  its local convexity. Structural aspects of the linearized equations,
 arising from the quasi-linearity of the full nonlinear problem, necessitate our imposing H\"older regularity  ($C^{2+2\alpha}$) in this energy argument. 
 Although conceptually elegant, no rate of decay emerges from our energy-based argument;
 $\mathcal{L}_-$, the linear operator generating the dynamics non-self-adjoint, and so it is the spectrum
 of the symmetric part of  $\mathcal{L}_-$, which we do not see how to access, which  holds
 the key to a decay rate by this method. 
\end{enumerate}

(II) \emph{Time-periodic nonlinear bubble oscillations for prescribed time-periodic far-field pressure
and their nonlinear asymptotic stability.} 
Consider the setting of a small-amplitude time-periodic far-field pressure, $P_\infty(t) = p_{\infty} + 
\psi(t;\om,A)$, with amplitude, $A$, and  frequency, $\om$.
Here $\psi(t;\om,A)$ is a function that is bounded and $2\pi/\om$- periodic in $t$. Further, for fixed $\om\in\R_+$, we assume that $\psi(t;\om,A) \to0$ as $A\to0$ uniformly for $t\in [0,2\pi/\om]$, e.g. $\psi(t;\om,A) = A\cos(\om t)$.
Theorem \ref{cor-period-exituniq} states that the system \eqref{red-eqns} admits a manifold of $2\pi/\omega$ time-periodic expanding and contracting bubble solutions, which parametrized by bubble mass $M$ and the pressure forcing amplitude $A$, which is taken to be sufficiently small.
Moreover, the $2\pi/\omega$ time-periodic solutions are nonlinearly and exponentially stable relative to small mass-preserving perturbations.
Theorem \ref{cor-period-exituniq} also shows the nonlinear asymptotic stability of the manifold of $2\pi/\omega$ time-periodic solutions and provides a lower bound on the decay exponent.
A comparison of our lower bound with the results of \cite{Prosperetti-JFM1991} is presented in Appendix \ref{sec-compare-prosperetti-per}.

Our results on asymptotic stability are in terms an $L^2-type$ norm of the perturbation. 
We remark that the geometric theory of semilinear parabolic equations (see, e.g. \cite{Henry-book1981})
 does not apply in the present setting since our equations are quasi-linear;
the types of maps which in  \cite{Henry-book1981} are contractions lose smoothness upon iteration. However, we can use the theory of quasi-linear parabolic PDEs to obtain global-in-time existence of unique solutions, for H\"older continuous data,  in higher-regular  H\"older  spaces.
Thus we can construct the Poincar\'e map and study the global-in-time dynamics near its fixed points (time-periodic solutions). The exponential time-decay estimates for the linearized evolution together with the above H\"older estimates enable us to prove exponential time-decay in the $L^2-type$ norm via the Duhamel's principle.
We note that one can bootstrap the $L^2$-exponential decay along with the H\"older estimates to get $C^{2+2\alpha}$-exponential decay via interpolation.

\subsection{Future directions}
The results of this article on nonlinear asymptotic stability of periodically pulsating bubbles for the model  \eqref{red-eqns}, and our earlier work \cite{LW-vbas2022} on spherical equilibrium bubbles,  hold only for the case of small-amplitude initial perturbations.
Indeed, this study relies on the local convexity of the total energy functional at an equilibrium, which can be used
to prove  first boundedness  \cite{BV-SIMA2000} and then time-decay \cite{LW-vbas2022} of solutions.
\emph{Is the equilibrium bubble nonlinearly asymptotically stable with respect to arbitrary, even  large, spherically symmetric perturbations? 
Do sufficiently large perturbations from equilibrium lead to a singularity within a finite time?}

A further interesting direction relates to the dynamics induced by a \underline{strong}  time-periodic  external (acoustic) pressure-field.  A low-dimensional ODE model for the nonlinear dynamics of the bubble radius, under a polytropic assumption for the gas, exhibits a complex sequence of bifurcations and chaotic motions
\cite{SBB:87}.    Do such behaviors arise in the PDE free-boundary model we have studied? One possibility is that, under strong forcing, the present model may exhibit singularity formation in finite time  or other complex behaviors.
In \cite{KP-JASA1989} the PDE model was studied numerically via a Galerkin method with the basis functions being polynomials, e.g. the shifted Chebyshev polynomials, even-ordered polynomials, and Legendre polynomials. (In contrast, our analysis makes use of Dirichlet eigenfunctions.)
Numerical evidence \cite[Figure 7]{KP-JASA1989} shows that bifurcations and chaos occur for large forcing pressure amplitude.

 Finally, analyzing the dynamics without the constraint of spherical symmetry would be of great mathematical and physical interest. For example, in the phenomenon of  \emph{sonoluminenscence} (see, for example, \cite{Brenner-sonoluminescence}), where
non-spherical deformations of the bubble surfaces ({\it shape modes}) are experimentally observed and play an important role. 
The analysis of the present and more accurate models without the constraint of spherical symmetry presents fascinating challenges.
In \cite{LW-vbas2022}, we presented a candidate for a more accurate model \cite[(2.1)-(2.4)]{LW-vbas2022} with which to consider this question.
Our study of spherically symmetric bubble dynamics in \cite{LW-vbas2022} and in the present paper could serve as a stepping stone toward the future study of asymmetric bubble dynamics, which is one of 
our motivations to work with the full PDE problem.
%

\subsection{Structure of the paper}\label{structure}

In Section \ref{linear-ev} we derive the linearized system of the unforced (with a constant far-field pressure) model \eqref{red-eqns-mass} around a fixed equilibrium.
We use a Dirichlet eigenfunction decomposition and the Laplace transform to obtain a representation formula for the solution, and prove that the solution exists for all time in the $L^2$ framework.
In Section \ref{pf-astab-norate} we present an independent study of the linearized system from an energy perspective. 
Using an energy method, we prove the asymptotic stability of the zero solution to the linearized system in the $C^{2+2\al}$ framework.
We then return to the study in the $L^2$ framework in Section \ref{sec-exp}, where
we use the Laplace transform representation to prove the exponential asymptotic stability of the zero solution, and present a rigorous estimate of the decay exponent.
In Section \ref{sec:forced-bubb} we study the time-periodically forced problem and prove the existence, uniqueness, and asymptotic stability of time-periodic solutions.
Appendix \ref{sec-compare-prosperetti} contains a detailed comparison of our rigorous lower bound on the decay exponent with the previous formal approximations \cite{Prosperetti-JFM1991}.

\subsection{Notation and conventions}

\begin{enumerate}
\item $B_R=\{x\in\mathbb{R}^3: |x|<R\}$
\item For a function $f(r)$ defined for $0<r<R$, we set $F(x) = f(|x|)$ for $x\in B_R$ and denote
$\norm{f}_{L^2_r(B_1)} = \norm{F}_{L^2_x(B_1)}$ and 
\EQ{\label{eq-holder-norm-def}
\norm{f}_{C^{2+2\al}_r} 
= \norm{F}_{C^{2+2\al}_x} 
= \max_{|\be|\le 2} \sup_{x\in B_R} |D^\be F(x)| + \sup_{\substack{x\neq y\\x,y\in B_R}} \frac{| D^2F(x) - D^2F(y) |}{|x-y|^{2\al}}.
}

\item Denote $\nb_r f = \nb_xF$ the derivative in radial direction and $\De_r f = \De_xF = \frac1{r^2}\pd_r(r^2\pd_rf)$ the radial part of the Laplace operator in $\mathbb R^3$.
\item For a state variable, such as the density $\rho$, if it corresponds to value of a constant equilibrium solution, then we denote it 
by $\rho_*$, and similarly for the values of other equilibrium state variables.
\item 
Throughout this paper, the following physical parameters and the equilibrium states
\[
\begin{array}{ccc}
\text{Thermal conductivity} & \cdots & \ka \\
\text{Thermal diffusivity} & \cdots & \chi \\
\text{Adiabatic constant} & \cdots & \ga  \\
\text{Heat capacity at constant volume} & \cdots &  c_v  \\
\text{Specific gas constant} & \cdots & \Rg \\
\text{Heat capacity at constant pressure} & \cdots &  c_p \\
\text{Far-field liquid temperature} & \cdots & T_\infty \\
\text{Far-field liquid pressure} & \cdots & p_{\infty,*} \\
\text{Surface tension} & \cdots & \si \\
\text{Dynamic viscosity of the liquid} & \cdots & \mu_l\\
\text{Density of the liquid} & \cdots & \rho_l\\
\text{Equilibrium gas density} & \cdots & \rho_*\\
\text{Equilibrium gas pressure} & \cdots & p_*\\
\text{Equilibrium bubble radius} & \cdots & R_*
\end{array}
\]
satisfy
\begin{equation}
c_p = \ga c_v = c_v + \Rg,\qquad
p_* = \Rg T_\infty \rho_* = p_{\infty,*} + \frac{2\si}{R_*},\qquad
\chi = \frac{\ka}{c_p\rho_*}.
\label{eq:params}\end{equation}

\item 
To a function $f(t)$, defined for $t\ge0$, we associate the Laplace transform and its inverse:
\begin{equation}\label{inversion}
\tilde{f}(\tau) = \int_0^\infty e^{-t\tau}f(t)dt,\
\text{ where }\quad
 f(t) = \frac{1}{2\pi i}\int_{\Gamma_a} e^{\tau t} \tilde{f}(\tau) d\tau.
 \end{equation}
Here, $\Gamma_a$ denotes the contour, $\Re\tau=a$,  with $a$ chosen so that $\tilde{f}(\tau)$ is analytic for $\Re\tau> a$.

 \item The normalized Dirichlet Laplacian eigenpairs for the unit ball $B_1\subset\mathbb{R}^3$
  are given by:
 \[ -\De\phi_j(x)=\la_j\phi_j(x),\quad \textrm{in $B_1$},\quad  \phi_j|_{|x|=1}=0,\quad 1=\int_{B_1} \phi_j^2(|x|)\, dx = 4\pi \int_0^1 \phi_j^2(y)y^2\, dy.\]
\EQ{\label{eq-eigen-formula}
\phi_j(|x|) = \frac{\sin(j\pi |x|)}{\sqrt{2\pi} |x|},\quad j=1,2,\ldots,\qquad
\la_j = (j\pi)^2,\quad j=1,2,\ldots.
}
\end{enumerate}

\subsection*{Acknowledgements} 
CL and MIW  are supported in part by the Simons Foundation Math + X Investigator Award \#376319 (MIW). MIW is also supported in part by  National Science Foundation Grant DMS-1908657 and DMS-1937254.
The authors thank Bj\"orn Sandstede for very helpful correspondence on the literature.

\section{The initial value problem for the linearized dynamics, and a representation of its solution}\label{linear-ev}

As proved by the authors in \cite{LW-vbas2022}, a small perturbation of an equilibrium spherical bubble, $(\rho_*,R_*)$,  evolves toward the equilibrium bubble determined by initial bubble mass. Since the manifold of equilibria, $\mathcal{M}_*$,  is one-dimensional, we expect the linearized dynamics to be exponentially contracting along a direction through $(\rho_*,R_*)\in \mathcal{M}_*$ which is transverse to  $\mathcal{M}_*$. To derive exponential time-decay of the linearized evolution we now introduce the linearized evolution constrained to this codimension one subspace.

It is convenient to change variables to a setting where the bubble of radius $R(t)$ is mapped 
to a fixed ball of radius one. Thus, we set $r=R(t)y$ and expand  \eqref{red-eqns} around any fixed equilibrium $(\rho_*,R_*)$:
\begin{equation}
\rho(r,t) = \rho_* + \ve \varrho(y,t) + O(\ve^2),\quad
R(t) = R_* + \ve \mathcal R(t) + O(\ve^2).\label{expandRhoR}
\end{equation} 
Retaining only terms which are of $O(\ve)$, we arrive at the linearized equations 
 governing infinitesimally small perturbations $(\varrho(y,t),\mathcal R(t))$:
\begin{subequations}
\label{red-eqns-linear}
\begin{align}
\pd_t\varrho(y,t) &= 
\frac{\chi}{R_*^2} \De_y\varrho(y,t) + \frac{1}{\ga}\,\pd_t \varrho(1,t),\quad 0\le y\le 1,\ t>0, \label{eq-bv-3.10prime-linear}\\
\dot{\mathcal R}(t) &= -\frac{\chi}{R_*\rho_*}\, \pd_y\varrho(1,t) - \frac{R_*}{3\ga\rho_*}\, \pd_t\varrho(1,t),\quad t>0,
\label{eq-bv-3.15prime-linear}\\
\varrho(1,t) &= \frac1{\Rg T_\infty} \bke{ - \frac{2\si}{R_*^2}\, \mathcal R 
+ \frac{4\mu_l}{R_*}\, \dot{\mathcal R}
+ \rho_lR_*\ddot{\mathcal R} },\quad t>0,
 \label{eq-bv-3.16prime-linear}
 \end{align}
\end{subequations}
where we have used the relation $\ga c_v = c_p$ (see \eqref{eq:params}). The parameter
\[ \chi = \frac{\ka}{\rho_*c_p}\]
 is the thermal diffusivity of the gas \cite{Prosperetti-JFM1991}. 
 
A  direct consequence of \eqref{eq-bv-3.10prime-linear} and \eqref{eq-bv-3.15prime-linear} is linearized conservation of mass:
\EQ{\label{eq-Dt-c_of_m-lin}
\frac{d}{dt}\left[ \int_{B_1} \varrho(|x|,t)\, dx +4\pi\, \frac{\rho_*}{R_*}\, \mathcal R(t)\right]=0. 
}
Conversely, \eqref{eq-Dt-c_of_m-lin} and \eqref{eq-bv-3.10prime-linear} imply \eqref{eq-bv-3.15prime-linear}.

Using $M[\rho,R]=M_*=(4\pi/3)\rho_*R_*^3$ and the expansion \eqref{expandRhoR},  we see that linearization of the constant mass constraint reads:
 \EQ{\label{eq-conserve-mass-linear}
 \int_{B_1} \varrho(|x|,t)\, dx +4\pi\, \frac{\rho_*}{R_*}\, \mathcal R(t)=0,\quad t>0. 
}
By\eqref{eq-Dt-c_of_m-lin},  we can ensure that \eqref{eq-conserve-mass-linear} holds for all $t>0$ by constraining the initial data
 for \eqref{red-eqns-linear} to satisfy:
 \EQ{\label{data-constraint}
 \int_{B_1} \varrho_0(|x|)\, dx +4\pi\, \frac{\rho_*}{R_*}\, \mathcal R_0=0. 
}
Hence, the linearized system \eqref{red-eqns-linear} with the initial data constraint \eqref{data-constraint}
 can be expressed as:
%
%
%
\begin{subequations}
\label{red-eqns-linear-mass-preserve}
\begin{align}
\pd_t\varrho(y,t) &= 
\frac{\chi}{R_*^2} \De_y\varrho(y,t) + \frac{1}{\ga}\,\pd_t \varrho(1,t),\quad 0\le y\le 1,\ t>0, \label{eq-bv-3.10prime-linear-mass-preserve}\\
\int_{B_1} \varrho(|x|,t)\, dx &= -4\pi\, \frac{\rho_*}{R_*}\, \mathcal R(t),\quad t>0,
\label{eq-conserve-mass-linear-mass-preserve}\\
\varrho(1,t) &= \frac1{\Rg T_\infty} \bke{ - \frac{2\si}{R_*^2}\, \mathcal R 
+ \frac{4\mu_l}{R_*}\, \dot{\mathcal R}
+ \rho_lR_*\ddot{\mathcal R} },\quad t>0,
 \label{eq-bv-3.16prime-linear-mass-preserve}
 \end{align}
\end{subequations}
with initial data $\varrho(|x|,0)$, $\mathcal R(0)$ and $\dot{\mathcal{R}}(0)$, chosen so that 
\eqref{eq-conserve-mass-linear-mass-preserve} holds at $t=0$.

With a view toward solving this system, we rewrite the diffusion PDE as one with  Dirichlet boundary conditions via the change of dependent variables:
\begin{equation} u(y,t)=\varrho(y,t) - \varrho(1,t),\quad  z(t) = \varrho(1,t),\label{uz-def}
\end{equation}
and note that on the boundary of $B_1$:
\[ u(1,t) = 0,\quad  \pd_yu(1,t) = \pd_y\varrho(1,t). \]
For convenience in the computations, we also introduce a composite parameter,
related to $\kappa$ and $\chi$:
\EQ{\label{eq-bar-ka-def}
\bar\ka \equiv \frac{\chi}{R_*^2} = \frac1{R_*^2} \frac{\ka}{\ga c_v\rho_*}.
}
In terms of the variables $u(y,t)$, $z(t)$ and $\mathcal{R}(t)$, the  linearized system \eqref{red-eqns-linear-mass-preserve} becomes
\begin{subequations}
\label{red-eqns-linear-new}
\begin{align}
\pd_tu &= \bar\ka \De_yu - \bke{ 1 - \frac1{\ga} } \dot z,\quad 0\le y\le1,\qquad u(1,t)=0,\quad \ t>0,\label{eq-u-pde}\\
z(t) &= \frac1{\Rg T_\infty} \bke{ -\frac{2\si}{R_*^2}\mathcal R + \frac{4\mu_l}{R_*}\, \dot{\mathcal R} + \rho_lR_*\ddot{\mathcal R} },\quad t>0,\label{eq-z}\\
\int_{B_1} u(|x|,t)\, dx &= -\frac{4\pi}3 z - 4\pi \frac{\rho_*}{R_*}\, \mathcal R,\quad t>0.\label{eq-u-mass-conserve} 
\end{align}
\end{subequations}
\begin{remark}
Equation \eqref{eq-bv-3.15prime-linear}, a consequence of \eqref{eq-bv-3.10prime-linear-mass-preserve}
 and \eqref{eq-conserve-mass-linear-mass-preserve}, 
may be written as 
\EQ{\label{eq-dotR-u}
\dot{\mathcal R} &= -\frac{\bar\ka R_*}{\rho_*}  \pd_yu(1,t) - \frac{R_*}{3\ga\rho_*}\, \dot z,\quad t>0.
}
\end{remark}

We consider the initial value problem for the system \eqref{red-eqns-linear-new} with initial data:
\begin{equation}\label{uzR-data} u(|x|,0)=u_0(|x|)\in L^2_y(B_1),\ \textrm{and}\ \quad z(0)=z_0,\  \mathcal{R}(0)=\mathcal{R}_0,\  \dot{\mathcal R}(0)=
\dot{\mathcal R}_0 \in\R\end{equation}
subject to the constraint:
\begin{equation}\label{uzR-constr}
\int_{B_1} u_0(|x|)\, dx = -\frac{4\pi}3 z_0 - 4\pi \frac{\rho_*}{R_*}\, \mathcal R_0.\end{equation}

Since $u(y,t)$ in \eqref{eq-u-pde}  is to satisfy homogeneous Dirichlet boundary conditions on $y=|x|=1$, we construct it as an expansion with respect to the eigenbasis for $L^2(B_1)$ of Dirichlet eigenfunctions of $-\Delta$.  The resulting infinite system of ODEs
can then be solved in terms of the Laplace transform; see \eqref{inversion}. The solution of the initial value problem is summarized in the following:

\begin{proposition}[Existence theory for the IVP \eqref{red-eqns-linear-new}, \eqref{uzR-data}, 
\eqref{uzR-constr}] \label{prop-asymp-stab-linear-u}
The IVP \eqref{red-eqns-linear-new}, \eqref{uzR-data}, 
\eqref{uzR-constr} 
has a unique global in time solution $t\mapsto(u(\cdot,t), z(t),\mathcal R(t))\in L^2_y(B_1)\times\R\times\R$
to the system \eqref{red-eqns-linear-new}.

The solution has the following representation:
\EQ{\label{eq-uzR}
u(y,t)= \sum_{j=1}^\infty c_j(t)\phi_j(y),\quad
z(t)= \mathcal L^{-1}\bket{\wt z(\tau)} (t),\quad
\mathcal R(t)= \mathcal L^{-1}\bket{\wt R(\tau)} (t),
}
where
\begin{align}
c_j(t) &= c_j(0) e^{-\bar\ka \la_j t} - \Ga_j \int_0^t e^{-\bar\ka \la_j(t-s)} \dot z(s)\, ds,\quad c_j(0) = \int_{B_1} u(|x|,0)\phi_j(|x|)\, dx, \label{eq-c}\\
\wt z(\tau) &= \frac1{\Rg T_\infty} \bkt{\bke{\rho_lR_*\tau^2 + \frac{4\mu_l}{R_*}\, \tau - \frac{2\si}{R_*^2}} \frac{\textrm{DATA}(\tau)}{Q(\tau)} - \rho_lR_*(\dot{\mathcal R}(0) + \tau\mathcal R(0))},\label{eq-tz}\\
\wt{\mathcal R}(\tau) &= \frac{\textrm{DATA}(\tau)}{Q(\tau)}.\label{eq-tR}
\end{align}
For $j\ge1$,  $\Ga_j$ is given by
\[
 \Ga_j = \frac{\ga-1}\ga \int_{B_1} \phi_j(|x|)\, dx,
\]
and $Q(\tau)$ is given by
\EQ{\label{eq-Q-def}
Q(\tau) = \frac1{\Rg T_\infty} \bke{\frac{4\pi}3 - \frac{8(\ga-1)}{\pi\ga} \sum_{j=1}^\infty \frac{\tau}{j^2\bke{\pi^2\bar\ka j^2 + \tau}} } \bke{\rho_lR_*\tau^2 + \frac{4\mu_l}{R_*}\, \tau - \frac{2\si}{R_*^2}} + 4\pi\,\frac{\rho_*}{R_*}.
}
Finally, $\textrm{DATA}(\tau)$ is determined by the initial conditions and is given by:
\EQ{\label{eq-DATA-def}
\textrm{DATA}(\tau) = - \frac{\ga}{\ga-1} \sum_{j=1}^\infty \Ga_j\, \frac{c_j(0) + \Ga_jz(0)}{\bar\ka\la_j+\tau} + \bke{\frac{4\pi}3 - \frac{\ga}{\ga-1} \sum_{j=1}^\infty \frac{\Ga_j^2 \tau}{\bar\ka\la_j + \tau}} \frac{\rho_lR_*(\dot{\mathcal R}(0) + \tau\mathcal R(0))}{\Rg T_\infty} .
}
\end{proposition}

\begin{remark}\label{poles_of_Q} Through the Laplace transform inversion formula, \eqref{inversion},  one sees that the 
location of the poles  of $Q$ determine the rate of exponential decay of solutions. That these poles lie in the open left half plane, is a consequence of energy dissipation; see for example Section \ref{pf-astab-norate}.
In Appendix \ref{beta-est} we perform a detailed analysis yielding bounds on the locations of the zeros of  $Q(\tau)$ in the left half complex $\tau-$ plane. These bounds then give a estimate for the exponentially decay rate for the linearized evolution; see Theorem \ref{thm-exp-decay-LT}. 
\end{remark}

\begin{remark}\label{rmk-nonlineaer-spectrum}
If we do not impose the linearized fixed mass constraint \eqref{uzR-constr}, there are non-decaying (neutral mode) solutions associated with manifold of equilibria.  Indeed, for the IVP \eqref{red-eqns-linear-new}, \eqref{uzR-data} (without the constraint \eqref{uzR-constr}) the proof of \cite[Proposition 9.3]{LW-vbas2022} shows that $\widetilde{\mathcal R}(\tau) = \textrm{DATA}(\tau)/(\tau Q(\tau))$ where $Q(\tau)$ is as in \eqref{eq-Q-def} and that $\wt z(\tau), c_j(t)$ are as in Proposition \ref{prop-asymp-stab-linear-u}.
In Proposition \ref{prop-asymp-stab-linear-u}, the linearized constant mass condition \eqref{eq-conserve-mass-linear}, which follows from the initial data constraint  \eqref{uzR-constr}, projects out the pole of the linearized operator resolvent at $\tau=0$; see \cite[9.30]{LW-vbas2022}.
The nonlinear picture is then consistent with exponential contraction onto a center manifold parametrized by bubble mass. See \cite[Section 9]{LW-vbas2022}.
\end{remark}

To prove Proposition \ref{prop-asymp-stab-linear-u}, in particular the derivation of the solution representation, 
we begin by reexpressing the system \eqref{red-eqns-linear-new} as an infinite-dimensional dynamical system.

\begin{proposition}[A dynamical system formulation]\label{prop-dyn-system-form}
Let $c_j = c_j(t)$ denote the $j$-th coefficients in the radial-Dirichlet-eigenfunction decomposition of $u$ as in \eqref{eq-uzR}.
Then, the system \eqref{red-eqns-linear-new} is equivalent to an infinite-dimensional dynamical system for ${\bf w} = (\mathcal R,\dot{\mathcal R}, c_1, c_2,\cdots)^\top$
\EQ{\label{eq-dyn-sys-unforced}
\dot{\bf w} = \mathcal L_-{\bf w},
}
The linear operator $\mathcal L_-$ is displayed below in \eqref{eq-def-L-}.
\end{proposition}

\begin{proof}
Using the Dirichlet eigenfunction expansion \eqref{eq-uzR} in \eqref{eq-u-pde} and \eqref{eq-u-mass-conserve} yields
\EQ{\label{eq-dotc-0}
\dot c_k = -\bar\ka\la_kc_k - \Ga_k\dot z,
}
\EQ{\label{eq-z-reduction}
z = -\frac3{4\pi} \frac{\ga}{\ga-1} \sum_{j=1}^\infty \Ga_jc_j - \frac{3\rho_*}{R_*} \mathcal R.
}
Using \eqref{eq-z-reduction} in \eqref{eq-dotc-0} and \eqref{eq-z}, we deduce
\EQ{\label{eq-dotc-1}
\dot c_k = -\bar\ka\la_kc_k - \Ga_k \bke{ -\frac{3\ga}{4\pi(\ga-1)} \sum_{j=1}^\infty \Ga_j \dot c_j - \frac{3\rho_*}{R_*} \dot{\mathcal R} },
}
\EQ{\label{eq-z-unforced-2}
\ddot{\mathcal R} &= \frac{\Rg T_\infty}{\rho_lR_*} \bke{  -\frac{3\ga}{4\pi(\ga-1)} \sum_{j=1}^\infty \Ga_j c_j - \frac{3\rho_*}{R_*} \mathcal R } + \frac{2\si}{\rho_lR_*^3}\mathcal R - \frac{4\mu_l}{\rho_lR_*^2}\dot{\mathcal R}\\
&= - b \mathcal R - \frac{4\mu_l}{\rho_lR_*^2} \dot{\mathcal R} - d \sum_{j=1}^\infty \Ga_jc_j,
}
where 
\EQ{\label{eq-def-b-d}
b := \frac{3p_{\infty,*}}{\rho_lR_*^2} + \frac{4\si}{\rho_lR_*^3} > 0,\quad \text{ and }\quad
d := \frac{3\ga\Rg T_\infty}{4\pi(\ga -1)\rho_lR_*} >0.
}
Above we used the relation $\Rg T_\infty \rho_* = p_{\infty,*} + 2\si/R_*$.

Thus, \eqref{eq-z-unforced-2} and \eqref{eq-dotc-1} form the infinite-dimensional dynamical system for ${\bf w} = (\mathcal R,\dot{\mathcal R}, c_1,c_2,\cdots)^\top$:
\EQN{
&\begin{pmatrix}
1&0&0&0&\cdots\\
0&1&0&0&\cdots\\
0&-\frac{3\rho_*}{R_*} \Ga_1&1-\frac{3\ga}{4\pi(\ga-1)}\Ga_1^2&-\frac{3\ga}{4\pi(\ga-1)}\Ga_1\Ga_2&\cdots\\
0&-\frac{3\rho_*}{R_*} \Ga_2&-\frac{3\ga}{4\pi(\ga-1)}\Ga_1\Ga_2&1-\frac{3\ga}{4\pi(\ga-1)}\Ga_2^2&\cdots\\
\vdots&\vdots&\vdots&\vdots&\ddots
\end{pmatrix}
\begin{bmatrix}
\mathcal R\\
\dot{\mathcal R}\\
c_1\\
c_2\\
\vdots
\end{bmatrix}^\prime\\
&\qquad\qquad\qquad\qquad\qquad\qquad\qquad=
\begin{pmatrix}
0&1&0&0&\cdots\\
-b & -\frac{4\mu_l}{\rho_lR_*^2}&-d\Ga_1&-d\Ga_2&\cdots\\
0&0&-\bar\ka\la_1&0&\cdots\\
0&0&0&-\bar\ka\la_2&\cdots\\
\vdots&\vdots&\vdots&\vdots&\ddots
\end{pmatrix}
\begin{bmatrix}
\mathcal R\\
\dot{\mathcal R}\\
c_1\\
c_2\\
\vdots
\end{bmatrix}.
}
Using $\sum_{j=1}^\infty\Ga_j^2 = 4(\ga-1)^2\pi/(3\ga^2)$, the inverse of the matrix on the left hand side above is 
\EQ{\label{eq-inverse-matrix}
&\begin{pmatrix}
1&0&0&0&\cdots\\
0&1&0&0&\cdots\\
0&-\frac{3\rho_*}{R_*} \Ga_1&1-\frac{3\ga}{4\pi(\ga-1)}\Ga_1^2&-\frac{3\ga}{4\pi(\ga-1)}\Ga_1\Ga_2&\cdots\\
0&-\frac{3\rho_*}{R_*} \Ga_2&-\frac{3\ga}{4\pi(\ga-1)}\Ga_1\Ga_2&1-\frac{3\ga}{4\pi(\ga-1)}\Ga_2^2&\cdots\\
\vdots&\vdots&\vdots&\vdots&\ddots
\end{pmatrix}^{-1}\\
&\qquad\qquad\qquad\qquad\qquad=
\begin{pmatrix}
1&0&0&0&\cdots\\
0&1&0&0&\cdots\\
0&\frac{3\rho_*}{R_*}\ga\Ga_1&1+\frac{3\ga^2}{4\pi(\ga-1)}\Ga_1^2&\frac{3\ga^2}{4\pi(\ga-1)}\Ga_1\Ga_2&\cdots\\
0&\frac{3\rho_*}{R_*}\ga\Ga_2&\frac{3\ga^2}{4\pi(\ga-1)}\Ga_1\Ga_2&1+\frac{3\ga^2}{4\pi(\ga-1)}\Ga_2^2&\cdots\\
\vdots&\vdots&\vdots&\vdots&\ddots
\end{pmatrix}.
}
Left-multiplying the inverse on both sides, we obtain
\EQ{\label{eq-matrix-form-unforced}
\begin{bmatrix}
\mathcal R\\
\dot{\mathcal R}\\
c_1\\
c_2\\
\vdots
\end{bmatrix}^\prime
&=
\begin{pmatrix}
0&1&0&0&\cdots\\
-b &-\frac{4\mu_l}{\rho_lR_*^2}&-d\Ga_1&-d\Ga_2&\cdots\\
-b \frac{3\rho_*}{R_*}\ga\Ga_1 & -\frac{12\mu_l\rho_*}{\rho_lR_*^3}\ga\Ga_1 &-\bar\ka\la_1 - e_1\Ga_1^2& -e_2\Ga_1\Ga_2&\cdots\\
-b \frac{3\rho_*}{R_*}\ga\Ga_2 & -\frac{12\mu_l\rho_*}{\rho_lR_*^3}\ga\Ga_2 &- e_1\Ga_1\Ga_2&-\bar\ka\la_2 -e_2\Ga_2^2&\cdots\\
\vdots&\vdots&\vdots&\vdots&\ddots
\end{pmatrix}
\begin{bmatrix}
\mathcal R\\
\dot{\mathcal R}\\
c_1\\
c_2\\
\vdots
\end{bmatrix},
}
where
\EQ{\label{eq-def-ej}
e_j := \bar\ka \la_j \frac{3\ga^2}{4\pi(\ga-1)} + \frac{3\rho_*}{R_*} \ga d.
}
Therefore, \eqref{eq-matrix-form-unforced} can be written in the form
\[
\dot{\bf w} = \mathcal L_- {\bf w},
\]
where
\EQ{\label{eq-def-L-}
\mathcal L_- = 
\begin{pmatrix}
0&1&0&0&\cdots\\
-b &-\frac{4\mu_l}{\rho_lR_*^2}&-d\Ga_1&-d\Ga_2&\cdots\\
-b \frac{3\rho_*}{R_*}\ga\Ga_1 & -\frac{12\mu_l\rho_*}{\rho_lR_*^3}\ga\Ga_1 &-\bar\ka\la_1 - e_1\Ga_1^2& -e_2\Ga_1\Ga_2&\cdots\\
-b \frac{3\rho_*}{R_*}\ga\Ga_2 & -\frac{12\mu_l\rho_*}{\rho_lR_*^3}\ga\Ga_2 &- e_1\Ga_1\Ga_2&-\bar\ka\la_2 - e_2\Ga_2^2&\cdots\\
\vdots&\vdots&\vdots&\vdots&\ddots
\end{pmatrix}.
}
in which $b,d$ are given in \eqref{eq-def-b-d}, and $e_j$ is defined in \eqref{eq-def-ej}.
This completes the proof of the proposition.
\end{proof}

We now study the linearized time-dynamics $\dot{\bf w} = \mathcal L_- {\bf w}$,  \eqref{eq-dyn-sys-unforced}.
Our main tool is  the Laplace transform.
Taking Laplace transform of the system \eqref{eq-dyn-sys-unforced},
one derives $(\mathcal L_- - \tau I)\wt{\bf w}(\tau) = -{\bf w}(0)$, where $I$ is the identity operator.
Consider the operator $\mathcal L_-$, acting in $\ell^2$,  with domain 
\[ \mathcal D(\mathcal L_-) = \{{\bf w}\in\ell^2:\mathcal L_-{\bf w}\in\ell^2\}.\]
The following result identifies  the spectrum of $\mathcal L_-$, 
\[
\si(\mathcal L_-) = \bket{\tau\in\CC: \mathcal L_- - \tau I\ \text{ is  not invertible as a  bounded operator from $\ell^2$ to $\mathcal D(\mathcal L_-)$}},
\]
 with the poles of $\wt{\bf w}(\tau)$.

\begin{proposition}\label{prop-spec-sec}
Let $\mathcal L_-$ be the linear operator defined in \eqref{eq-def-L-}.
Then

(1)
\[
\si(\mathcal L_-) = \{\tau\in\CC: Q(\tau) = 0\},
\]
where $Q(\tau)$ is given in \eqref{eq-Q-def}.

(2) There exists $\be>0$ such that any $\tau\in\mathbb{C}$ and $Q(\tau) = 0$ implies 
 $\Re(\tau)\le -\be< 0$, where a lower bound for $\be$ is displayed in  \eqref{eq-be-sharp-def-appendix}
  of Appendix \ref{beta-est}.

(3) Moreover, there exists a constant $C=C(\be)$ such that $\norm{(\mathcal L_- - \tau I)^{-1}} \le C(\be)$ for all $\tau\in\CC$ with $\Re(\tau)\le-\be$.

(4) The operator $-\mathcal L_-$ is sectorial.
In particular, $\mathcal L_-$ generates an analytic semigroup $\bket{e^{\mathcal L_- t} }_{t\ge0}$ with
\EQ{\label{eq-semigroup-decay-est}
\norm{e^{\mathcal L_- t}} \le C e^{-\be t},\quad t\ge0,
}
for some $C>0$.
\end{proposition}

\begin{proof}
Expand $u$ in terms of the radial Dirichlet eigenfunctions as
\EQ{\label{eq-u-expansion}
u(y,t) = \sum_{j=1}^\infty c_j(t) \phi_j(y),
}
where $\phi_j$, $j=1,2,\ldots$, are defined in \eqref{eq-eigen-formula}.
Plugging \eqref{eq-u-expansion} into \eqref{eq-u-pde},
\EQ{\label{eq-u-fourier}
\sum_{j=1}^\infty \dot c_j(t) \phi_j(y) = -\bar\ka \sum_{j=1}^\infty \la_jc_j(t)\phi_j(y) - \frac{\ga-1}\ga \dot z(t).
}
Taking inner product of \eqref{eq-u-fourier} in $L^2(B_1)$ with $\phi_k(y)$, $k=1,2,\ldots$, one has
\EQ{\label{eq-dotc}
\dot c_k(t) = -\bar\ka\la_kc_k(t) - \Ga_k\dot z(t),
}
where
\EQ{\label{eq-Ga-formula}
\Ga_k = \frac{\ga-1}\ga \int_{B_1} \phi_k(|x|)\, dx
= \frac{2\sqrt2(\ga-1)}{\sqrt\pi \ga}\, \frac{(-1)^{k-1}}k.
}
This implies \eqref{eq-c} of Proposition \ref{prop-asymp-stab-linear-u}.

Taking Laplace transform of \eqref{eq-dotc}
, we have
\[
- c_k(0) + \tau\widetilde{c_k}(\tau) 
= - \bar\ka\la_k\widetilde{c_k}(\tau) - \Ga_k(-z(0) + \tau\widetilde z(\tau)),
\]
or
\EQ{\label{eq-tdc}
\widetilde{c_k}(\tau) = \frac{c_k(0) + \Ga_kz(0)}{\bar\ka\la_k+\tau} - \frac{\Ga_k\tau}{\bar\ka\la_k+\tau}\, \widetilde z(\tau).
}

Taking Laplace transform of \eqref{eq-u-mass-conserve}
, we obtain
\[
\sum_{j=1}^\infty \widetilde{c_j}(\tau) \int_{B_1}\phi_j\, dx = - \frac{4\pi}3\, \widetilde z(\tau) -4\pi\, \frac{\rho_*}{R_*}\, \widetilde{\mathcal R}(\tau),
\]
or, by using \eqref{eq-tdc} and \eqref{eq-Ga-formula},
\[
\sum_{j=1}^\infty \bke{\frac{c_j(0) + \Ga_jz(0)}{\bar\ka\la_j+\tau} -\frac{\Ga_j\tau}{\bar\ka\la_j + \tau}\, \widetilde z(\tau) }\cdot \frac{\ga}{\ga-1}\, \Ga_j  
= - \frac{4\pi}3\, \widetilde z(\tau) - 4\pi\, \frac{\rho_*}{R_*}\, \widetilde{\mathcal R}(\tau).
\]
Rearranging the above, we deduce
\EQ{\label{eq-tdR}
\bke{\frac{4\pi}3 - \frac{\ga}{\ga-1} \sum_{j=1}^\infty \frac{\Ga_j^2 \tau}{\bar\ka\la_j + \tau}} \widetilde z(\tau) + 4\pi\,\frac{\rho_*}{R_*}\, \widetilde{\mathcal R}(\tau) = - \sum_{j=1}^\infty \frac{c_j(0) + \Ga_jz(0)}{\bar\ka\la_j+\tau}\cdot \frac{\ga}{\ga-1}\, \Ga_j.
}

Taking Laplace transform of \eqref{eq-z}
, we derive
\[
\Rg T_\infty \widetilde z(\tau) = -\frac{2\si}{R_*^2} \widetilde R(\tau) + \frac{4\mu_l}{R_*} \bke{-\mathcal R(0) + \tau \widetilde R(\tau) } + \rho_lR_*\bke{-\dot{\mathcal R}(0) - \tau\mathcal R(0) + \tau^2\widetilde R(\tau)},
\]
or
\EQ{\label{eq-tdz}
\bke{\rho_lR_*\tau^2 + \frac{4\mu_l}{R_*}\, \tau - \frac{2\si}{R_*^2}} \widetilde{\mathcal R}(\tau) - \Rg T_\infty\widetilde z(\tau) 
= \rho_lR_* \bke{\dot{\mathcal R}(0) + \tau\mathcal{R}(0)},
}
which implies \eqref{eq-tz} of  Proposition \ref{prop-asymp-stab-linear-u}.

Replacing $\widetilde z(\tau)$ in \eqref{eq-tdR} using \eqref{eq-tdz},
\[
\bkt{ \frac1{\Rg T_\infty} \bke{\frac{4\pi}3 - \frac{\ga}{\ga-1} \sum_{j=1}^\infty \frac{\Ga_j^2 \tau}{\bar\ka\la_j + \tau}} \bke{\rho_lR_*\tau^2 + \frac{4\mu_l}{R_*}\, \tau - \frac{2\si}{R_*^2}} + 4\pi\,\frac{\rho_*}{R_*} } \widetilde R(\tau) = \textrm{DATA}(\tau),
\]
where $\textrm{DATA}(\tau)$ is given in \eqref{eq-DATA-def} whose poles located at $-\bar\ka\la_j = -\pi^2\bar\ka j^2$, $j=1,2,\ldots$.
Using \eqref{eq-Ga-formula} and \eqref{eq-eigen-formula} we have
\[
\frac{\Ga_j^2\tau}{\bar\ka\la_j + \tau} = \frac{8(\ga - 1)^2}{\pi\ga^2}\, \frac{\tau}{j^2(\pi^2\bar\ka j^2 + \tau)},
\]
and thus,
\[
Q(\tau) \widetilde R(\tau) = \textrm{DATA}(\tau),
\]
where $Q(\tau)$ is defined in \eqref{eq-Q-def}. This proves \eqref{eq-tR} of Proposition \ref{prop-asymp-stab-linear-u}.

Using \eqref{eq-tdz} in \eqref{eq-tdc} to replace $\widetilde z(\tau)$, $\widetilde{c_k}(\tau)$ can be written in terms of $\widetilde R(\tau)$, hence $Q(\tau)$, and the initial data.
The expressions of $\widetilde R(\tau)$ and $\widetilde{c_k}(\tau)$ in terms of the initial data and $Q(\tau)$ amounts to, taking Laplace transform of \eqref{eq-dyn-sys-unforced}, $\widetilde{\bf w}(\tau) = (\mathcal L_- - \tau I)^{-1}\widetilde{\bf w}(0)$ for all $\tau\in\CC$ with $Q(\tau)\neq0$.
The proposition then follows immediately from Lemmas \ref{lem-negative-upper-bound} and \ref{lem-sectorial-bound} and the relation between sectorial operators and analytic semigroups (see e.g. \cite[Theorem 1.3.4, p.20]{Henry-book1981}).
\end{proof}

We now prove the existence of global in time solution in $L^2$ setting and complete the proof of Proposition \ref{prop-asymp-stab-linear-u}.

\begin{proof}[Proof of Proposition \ref{prop-asymp-stab-linear-u}]
The solution formula \eqref{eq-uzR}-\eqref{eq-tR} is established in the proof of Proposition \ref{prop-spec-sec}.
It suffices to show the global in time existence of the unique solution $(u,z,\mathcal R)$ of \eqref{red-eqns-linear-new} in $L^2$ setting.
By Proposition \ref{prop-dyn-system-form}, it is equivalent to the existence of unique global in time solution ${\bf w}(t)$ of $\dot{\bf w} = \mathcal L_-{\bf w}$ in $\ell^2$. 
Since $\{e^{\mathcal L_- t}\}_{t\ge0}$ is an analytic semigroup defined on $\ell^2$ by Part (3) of Proposition \ref{prop-spec-sec},
the desired existence and unique of the solution ${\bf w}\in\ell^2$ then follows from the classical semigroup theory (see eg. \cite[Corollary 1.5, p. 104]{Pazy-book1983}).
\end{proof}

In the following section, we present a ``weak form'' of time-decay of solutions
of the  IVP \eqref{red-eqns-linear-new}, \eqref{uzR-data}, 
\eqref{uzR-constr}; decay is proved with no rate deduced. This proof makes use
of a linearized energy dissipation law and not on a detailed analysis of $Q(\tau)$.

\section{Linear asymptotic stability of spherical equilibrium bubbles by an energy method}\label{pf-astab-norate}

In this section we show that the zero solution of the linearized system \eqref{red-eqns-linear-mass-preserve},
subject to the linearized constant mass (data) constraint \eqref{data-constraint},  is asymptotically stable by an analogous energy approach to that carried out in \cite{LW-vbas2022} for the nonlinear problem \eqref{red-eqns}. 
Introduce a norm for measuring the size of the perturbation in the $C^{2+2\al}$ setting: 
For $\al\in(0,1/2)$,
\EQ{\label{eq-def-oldnorm-holder}
\oldnorm{\left(\varrho(\cdot,t),\mathcal R(t),\dot{\mathcal R}(t)\right)}_{C^{2+2\al}} 
  \equiv \| \varrho(\cdot,t) \|_{C^{2+2\al}_y(B_1)} + |\mathcal R(t)| + \abs{\dot{\mathcal R}(t)}.
}
Here, $\norm{f(\cdot,t)}_{C^{2+2\al}_y(B_1)}$ is given by \eqref{eq-holder-norm-def}.

\begin{theorem}\label{thm-asymp-stab-linear}
Consider the initial value problem for \eqref{red-eqns-linear-mass-preserve} with $C^{2+2\al}_y$ initial data $\varrho(\cdot,0)$, $\al\in(0,1/2)$, satisfying the linearized constant mass (data) constraint \eqref{data-constraint}.
Then

(1) there exists a unique global in time solution $(\varrho,\mathcal R)\in C^{2+2\al,1+\al}_{y,t}(B_1\times[0,\infty))\times C^{3+\al}_t[0,\infty)$, and

(2)
 \[
\oldnorm{\left(\varrho(\cdot,t),\mathcal R(t),\dot{\mathcal R}(t)\right)}_{C^{2+2\al}}  \to0\ \text{ as }t\to+\infty.
 \]

Furthermore, we have $\ddot{\mathcal R}(t)$, $\dddot{\mathcal R}(t)\to0$ as $t\to\infty$.
\end{theorem}

The nonlinear asymptotic stability of the equilibrium solutions of \eqref{red-eqns} has been established in \cite{LW-vbas2022} by using an energy dissipation law.
To prove Theorem \ref{thm-asymp-stab-linear} we adopt such an approach based on an appropriate choice of  linearized energy. In view of the quadratic terms in the expansion \cite[(4.27)]{BV-SIMA2000} of the energy for the nonlinear system \eqref{red-eqns}, we define $\mathcal E_{\rm total}^L$, the {\it total energy for the linearized system} \eqref{red-eqns-linear-new}.

\begin{definition}[The total energy]
The total energy of the linearized system \eqref{red-eqns-linear-new} is given by 
\EQ{\label{eq-linear-energy}
\mathcal E_{\rm total}^L &= \mathcal E_{\rm total}^L[u,z,\mathcal R,\dot{\mathcal R}]\\
& = -4\pi\si\mathcal R^2 - 4\pi\Rg T_\infty R_*^2 \mathcal R z - \frac{2\pi\Rg T_\infty R_*^3}{3\rho_*}\,  z^2 + \frac{c_v\ga T_\infty R_*^3}{2\rho_*}  \int_{B_1} u^2\, dx + 2\pi\rho_lR_*^3\dot{\mathcal R}^2.
}
\end{definition}
The following formula for the total energy, $\mathcal E_{\rm total}^L$, is useful for deriving the convexity inequality and  
 a positive lower bound of $\mathcal E_{\rm total}^L$.

\begin{proposition}\label{two-en}
For any $(u,z,\mathcal R)$ satisfying the linearized constant mass condition \eqref{eq-u-mass-conserve},
the linearized total energy, $\mathcal E_{\rm total}^L$,  can be written as the following expression:
\EQ{\label{eq-energy-equation}
\mathcal E_{\rm total}^L &= 2\pi (4\si + 3p_{\infty,*}R_*) \mathcal R^2 - \frac{R_*^2}{8\pi\rho_*^2}\, (6\si + 3p_{\infty,*}R_*) \bke{\int_{B_1} u\, dx}^2 + \frac{c_v\ga T_\infty R_*^3}{2\rho_*}  \int_{B_1} u^2\, dx + 2\pi\rho_lR_*^3\dot{\mathcal R}^2.
}
\end{proposition}

\begin{proof}
Firstly, \eqref{eq-u-mass-conserve} gives
\EQN{
z = -\frac3{4\pi} \int_{B_1} u\, dx - \frac{3\rho_*}{R_*}\, \mathcal R.
}
Substituting $z$ in \eqref{eq-linear-energy} by the above, 
\EQN{
\mathcal E_{\rm total}^L 
&= -4\pi\si\mathcal R^2 - 4\pi\Rg T_\infty R_*^2 \mathcal R \bke{ -\frac3{4\pi} \int_{B_1} u\, dx - \frac{3\rho_*}{R_*}\, \mathcal R } - \frac{2\pi\Rg T_\infty R_*^3}{3\rho_*} \bke{ -\frac3{4\pi} \int_{B_1} u\, dx - \frac{3\rho_*}{R_*}\, \mathcal R }^2\\
&\quad + \frac{c_v\ga T_\infty R_*^3}{2\rho_*}  \int_{B_1} u^2\, dx + 2\pi\rho_lR_*^3\dot{\mathcal R}^2\\
&= \bke{-4\pi\si + 6\pi\Rg T_\infty R_*\rho_*} \mathcal R^2 - \frac{2\pi\Rg T_\infty R_*^3}{3\rho_*} \bke{\frac3{4\pi} \int_{B_1} u\, dx}^2 + \frac{c_v\ga T_\infty R_*^3}{2\rho_*}  \int_{B_1} u^2\, dx + 2\pi\rho_lR_*^3\dot{\mathcal R}^2.
}
Then \eqref{eq-energy-equation} follows from the constitutive law $\Rg T_\infty\rho_* = p_{\infty,*} + 2\si/R_*$.
\end{proof}

\subsection{Linearized energy dissipation identity}

The energy functional $\mathcal E_{\rm total}^L$ in \eqref{eq-linear-energy} is important since it obeys the following energy dissipation law.

\begin{proposition}[Energy dissipation identity]\label{prop-energy-dissipation}
Assume $(u,z,\mathcal R)$ is a solution of the linearized system \eqref{red-eqns-linear-new}. Then
\EQ{\label{eq-energy-dissipation}
\frac{d}{dt} \mathcal E_{\rm total}^L = - \frac{\ka T_\infty}{\rho_*^2}\, R_* \int_{B_1} |\nb_y u|^2\, dx -16\pi\mu_lR_*\dot{\mathcal R}^2.
}
\end{proposition}
\begin{proof} 
Recall that the constants $\ka$ and $\bar\ka$ are related in \eqref{eq-bar-ka-def}.
Multiplying \eqref{eq-u-pde} by $u$, integrating over $B_1$ and using \eqref{eq-u-mass-conserve}, we get
\EQ{\label{eq-u-energy-dynamic}
\int_{B_1} \pd_tu\cdot u\,dx &= -\bar\ka \int_{B_1} |\nb_y u|^2\, dx - \bke{1-\frac1\ga} \dot z \int_{B_1} u\, dx\\
&= -\bar\ka \int_{B_1} |\nb_y u|^2\, dx + \bke{1-\frac1\ga} \frac{4\pi}3\, \dot z z +  \bke{1-\frac1\ga}\, 4\pi\, \frac{\rho_*}{R_*}\, \dot z \mathcal R.
}
Multiplying \eqref{eq-z} by $\Rg T_\infty\dot{\mathcal R}$, one has
\EQ{\label{eq-zRdot}
-\frac{2\si}{R_*^2}\, \mathcal R \dot{\mathcal R} + \frac{4\mu_l}{R_*}\, \dot{\mathcal R}^2 + \rho_lR_* \ddot{\mathcal R} \dot{\mathcal R}
 = \Rg T_\infty z\dot{\mathcal R}.
}

Multiplying \eqref{eq-u-energy-dynamic} by $\frac{c_v\ga T_\infty}{\rho_*}\, R_*^3$ and \eqref{eq-zRdot} by $4\pi R_*^2$, and adding them up,
\EQN{
\frac{c_v\ga T_\infty}{\rho_*}\, R_*^3 \int_{B_1}& \pd_tu\cdot u\, dx - 8\pi\si\mathcal R \dot{\mathcal R} + 16\pi\mu_lR_*\dot{\mathcal R}^2 + 4\pi\rho_lR_*^3 \ddot{\mathcal R} \dot{\mathcal R} \\
&= -\bar\ka\, \frac{c_v\ga T_\infty}{\rho_*}\, R_*^3 \int_{B_1} |\nb_y u|^2\, dx + \bke{1-\frac1\ga} \frac{4\pi}3\, \frac{c_v\ga T_\infty}{\rho_*}\, R_*^3 \dot z z \\
&\quad + \bke{1-\frac1\ga} 4\pi\, \frac{\rho_*}{R_*}\, \frac{c_v\ga T_\infty}{\rho_*}\, R_*^3 \dot z\mathcal R + 4\pi\Rg T_\infty R_*^2 z \dot{\mathcal R}.
}
Using the relations of the parameters: \eqref{eq-bar-ka-def} and $c_v(\ga-1) = \Rg$, we deduce, after a rearrangement, from the above equation that
\EQN{
- \frac{\ka T_\infty}{\rho_*^2}\, R_* \int_{B_1} |\nb_y u|^2\, dx -16\pi\mu_lR_*\dot{\mathcal R}^2
&= - \frac{4\pi\Rg T_\infty}{3\rho_*}\, R_*^3 \dot z z - 4\pi \Rg T_\infty R_*^2 \dot z\mathcal R - 4\pi\Rg T_\infty R_*^2 z \dot{\mathcal R}\\
&\quad + \frac{c_v\ga T_\infty}{\rho_*}\, R_*^3 \int_{B_1} \pd_tu\cdot u\, dx - 8\pi\si\mathcal R \dot{\mathcal R} + 4\pi\rho_lR_*^3 \ddot{\mathcal R} \dot{\mathcal R} .
}
This implies the energy dissipation law.
\end{proof}

\subsection{Convexity and positivity of the linearized energy}\label{sec:convexity}


At the heart of the nonlinear stability analysis of  \eqref{red-eqns} in \cite{LW-vbas2022} 
is the local convexity of the total energy around the equilibrium $(\rho_*[M],\rho_*[M])$, relative to small perturbations of
 mass, $M$. 
Here we note the convexity of the linearized energy, associated with  \eqref{red-eqns-linear-new}.

\begin{proposition}[Convexity and positivity of the linearized energy]\label{prop-coercivity}
For any $(u,z,\mathcal R)$ we have 
\EQ{\label{eq-coercivity}
\mathcal E_{\rm total}^L \ge 2\pi (4\si + 3p_{\infty,*}R_*) \mathcal R^2 + \frac{c_vT_\infty R_*^3}{2\rho_*} \int_{B_1} u^2\, dx
+ 2\pi\rho_lR_*^3\dot{\mathcal R}^2.
}
\end{proposition}

\begin{proof}
By the linearized constant mass constraint \eqref{eq-u-mass-conserve}, the total energy can be written as in \eqref{eq-energy-equation}. 
Using the Cauchy-Schwarz inequality, $\bke{\int_{B_1} u}^2 \le |B_1| \int_{B_1} u^2 = \frac{4\pi}3 \int_{B_1} u^2$ in \eqref{eq-energy-equation},
\EQN{
\mathcal E_{\rm total}^L &\ge 2\pi (4\si + 3p_{\infty,*}R_*) \mathcal R^2 + \bkt{ - \frac{R_*^2}{8\pi\rho_*^2}\, (6\si + 3p_{\infty,*}R_*)\, \frac{4\pi}3 + \frac{c_v\ga T_\infty R_*^3}{2\rho_*}  } \int_{B_1} u^2\, dx\\
&\quad + 2\pi\rho_lR_*^3\dot{\mathcal R}^2\\
&= 2\pi (4\si + 3p_{\infty,*}R_*) \mathcal R^2 + \frac{c_vT_\infty R_*^3}{2\rho_*} \int_{B_1} u^2\, dx
+ 2\pi\rho_lR_*^3\dot{\mathcal R}^2,
}
where we've used $c_v\ga = c_v + \Rg$, $\Rg T_\infty\rho_* = p_{\infty,*} + 2\si/R_*$. 
This completes the proof of Proposition \ref{prop-coercivity}. \end{proof}

\subsection{Outline of the Proof of Theorem \ref{thm-asymp-stab-linear}}\label{proof-thm-asymp-stab-linear}

We start by proving Part (1), existence of unique global in time solution in $C^{2+2\al}$.
The first step is to apply a similar Leray-Schauder fixed point argument in H\"older spaces to that  used in the local well-posedness proof of \cite[Theorem 3.1]{BV-SIMA2000}. 
In the proof of \cite[Theorem 3.1]{BV-SIMA2000}, the classical regularity theory for quasilinear parabolic equations is needed. Here, we only need to use the regularity theory for linear parabolic equations.
Next, using the convexity inequality \eqref{eq-coercivity} and an interpolation inequality, one can show the global existence of solutions $(u,z,\mathcal R)$ in H\"older spaces and derive a uniform bound following from the same bootstrap argument in \cite[Theorem 4.1]{BV-SIMA2000}.

Next, we show the convergence stated in Part (2).
Integrating the  energy dissipation identity of Proposition \ref{prop-energy-dissipation} we have for any $T>0$:
\begin{align*}
\mathcal E_{\rm total}^L(T) + \frac{\ka T_\infty}{\rho_*^2}\, R_* \int_0^T \ \int_{B_1} |\nb_y u|^2\, dx\ ds + 16\pi\mu_lR_*\int_0^T \dot{\mathcal R}^2(s) ds =  \mathcal E_{\rm total}^L (0)  
\end{align*}
By the convexity and positivity inequality of Proposition \ref{prop-coercivity}, 
\EQ{\label{eq-unif-bound}
 \int_0^T \ \int_{B_1} |\nb_y u|^2\, dx\ ds  \quad \textrm{and}\quad 
\mu_l \int_0^T \dot{\mathcal R}^2(s) ds 
}
are uniformly bounded  as functions of $T\in\mathbb{R}_+$ in terms of the initial data. With this observation as a starting point, we apply a similar argument to that in \cite{LW-vbas2022} (for the nonlinear system \eqref{red-eqns}),  obtain the time-decay results asserted in Theorem \ref{thm-asymp-stab-linear}.


The uniform bound  \eqref{eq-unif-bound} and the regularity properties of the $u$ and $\mathcal R$
(e.g. that $\int_{B_1} |\nb_yu|^2\, dx$ and $\dot{\mathcal R}^2$ are uniformly continuous functions in $t$)
 can be used together with Barbalat's lemma as a point of departure for proving the time decay of $u$ and $\mathcal R$ in higher norms.
 See the proof of \cite[Proposition 8.1]{LW-vbas2022}.

\section{Exponential time-decay of linearized dynamics}\label{sec-exp}
This section is devoted to the proof of exponential time-decay of the solution of the IVP 
for  the linearized system, with data constrained by the linearized fixed mass constraint.

Introduce a norm for measuring the size of the perturbation in the $L^2$ setting
\EQ{\label{eq-def-oldnorm}
\oldnorm{\left(\varrho(\cdot,t),\mathcal R(t),\dot{\mathcal R}(t)\right)}_{L^2} 
  \equiv \| \varrho(\cdot,t) \|_{L^2_y(B_1)} + |\mathcal R(t)| + \abs{\dot{\mathcal R}(t)}.
}
We now state the main exponential stability theorem for the linearized system \eqref{red-eqns-linear-mass-preserve}.

\begin{theorem}\label{thm-exp-decay-LT}
The initial value problem for \eqref{red-eqns-linear-mass-preserve} with $\chi>0$ (or equivalently $\kappa>0$), starting from a $L^2$ initial data $\varrho(\cdot,0)$ satisfying the linearized constant mass constraint \eqref{data-constraint},
has a unique global in time solution $t\mapsto(\varrho(\cdot,t),\mathcal R(t))\in L^2_y(B_1)\times\R$.
Moreover,
there is a constant  $\beta>0$ such that the following holds:
\begin{enumerate}
\item Given initial conditions which satisfy \eqref{uzR-data}, there is a constant $C_0$, depending on the initial data, 
such that for all $t>0$:
Then,  
\[\oldnorm{\left(\varrho(\cdot,t),\mathcal R(t),\dot{\mathcal R}(t)\right)}_{L^2} \le C_0\,  e^{-\be t},\]
where the norm is defined in  \eqref{eq-def-oldnorm}.

 \item The exponential decay rate constant, $\beta$, having units $\textrm{[time]}^{-1}$,  can be taken to be:
\EQ{\label{eq-be-sharp-def}
\be &= \min\Bigg\{ \bke{1 - \sqrt{\dfrac{\vartheta(\ga)}{ \frac{p_{\infty,*}R_*}{2p_{\infty,*}R_*+6\si} + \vartheta(\ga) }}}\pi^2\bar\ka,\  \sqrt{ \e \frac{2p_*}{\rho_lR_*^2} },\\
&\qquad \frac{2\mu_l}{\rho_lR_*^2} + \mathbbm{1}_{\De\le0}\, \frac{(1-\ve)^2\vartheta(\ga)p_*}{\pi^4\bar\ka \rho_lR_*^2} \bke{\frac{4\pi^4}{90} + (1-\ve)^{3/4} O\bke{\bke{\frac1{\pi^2\bar\ka}\sqrt{ \frac{2p_*}{\rho_lR_*^2 }}}^{3/2}}} - \mathbbm{1}_{\De>0}\, \frac{\sqrt{\De}}{2\rho_lR_*} \Bigg\}.
}
Here, $\bar\ka = \chi/R_*^2$, $\vartheta(\ga)=1-\ga^{-1}$, $\ve\in(0,1)$ is arbitrary, and $\De:= \bke{ \frac{4\mu_l}{R_*} }^2 - 8\rho_lp_*$.
\end{enumerate}
\end{theorem}

\begin{remark}\label{rmk-higher-norm-conv}
%
Suppose the initial data $\varrho(\cdot,0)$ is in $C^{2+2\al}_y(B_1)$.
By a similar argument as in the proof of Theorem \ref{thm-asymp-stab-linear} through the Barbalat's lemma and the interpolation inequality in \cite[Lemma D.1]{LW-vbas2022}, one can bootstrap the $L^2$-exponential decay to $C^{2+2\al}$-exponential decay. That is, one have
\[
\norm{\varrho(\cdot,t)}_{C^{2+2\al}_y(B_1)},\ \ddot{\mathcal R}(t),\ \dddot{\mathcal R}(t) = O(e^{-\de\be t})\ \text{ as }\ t\to\infty,
\]
for some $\de<1$, where $\be$ as in \eqref{eq-be-sharp-def}.
\end{remark}



\begin{proof}[Proof of Theorem \ref{thm-exp-decay-LT}]
First note that, via a change of variables, the initial value problem \eqref{red-eqns-linear-mass-preserve} for $(\varrho,\mathcal R)$ is equivalent to the system \eqref{red-eqns-linear-new} for $(u,z,\mathcal R)$.
Since $\varrho(\cdot,0)\in L^2_y(B_1)$, $u(\cdot,0)\in L^2_y(B_1)$.
By Proposition \ref{prop-asymp-stab-linear-u}, there exists a unique global in time solution $(u,z,\mathcal R)\in L^2_y(B_1)\times\R\times\R$ of the system \eqref{red-eqns-linear-new}.
Further converting the system \eqref{red-eqns-linear-new} to the infinite-dimensional dynamical system \eqref{eq-dyn-sys-unforced}, $\dot{\bf w} = \mathcal L_-{\bf w}$, in Proposition \ref{prop-dyn-system-form},
and using the semigroup decay estimate \eqref{eq-semigroup-decay-est},
we deduce that $\norm{ {\bf w}(t) }\le C e^{-\be t}$ for some $C>0$ depending on the initial data ${\bf w}(0)$, where $\be>0$ can be taken to be as in \eqref{eq-be-sharp-def-appendix}.
Converting ${\bf w}(t)$ back to the state variables $\varrho(\cdot,t)$ and $\mathcal R(t)$, we complete the proof of Theorem \ref{thm-exp-decay-LT}.
\end{proof}


\section{Time-periodic bubble oscillations of the nonlinear bubble-liquid system,
 and their nonlinear asymptotic stability}
\label{sec:forced-bubb}

In this section we consider the periodically forced nonlinear bubble-fluid system \eqref{red-eqns} with $2\pi/\omega$- time periodic pressure field at infinity:
\EQ{\label{eq-periodic-forcing-0}
p_\infty(t) = p_{\infty,*} + \psi(t;\om,A),
}
where $\psi(t;\om,A)$ is $2\pi/\om$- periodic and bounded in $t$ such that, for any give frequency $\om\in\R_+$, $\psi(t;\om,A)\to0$ as the amplitude $A\to0$.
An example is $\psi(t;\om,A)=A\cos(\om t)$.
For the time-periodically forced model, we prove the existence and stability of time-periodic  bubble oscillations.
The idea of the proof is to determine the initial condition, which gives rise to a $2\pi/\om$- time periodic solution by finding fixed points of a Poincar\'e map; see the classical strategy discussed in \cite[Chapter 8]{Henry-book1981} for semilinear dynamical systems.
Due to the quasilinear character of our free-boundary PDE model, an a priori regularity estimate for the solution is required.
To this end, we first derive an equivalent infinite-dimensional dynamical system in Section \ref{sec-dyn-sys-derivation}. 
With the aid of the dynamical system, we are able to estimate the nonlinear term and establish the global well-posedness for the periodically forced problem near the equilibrium of the unforced problem in the $C^{2+2\al}$ framework (see Section \ref{sec-global-wellp}).
The two key points in our analysis are the spectral analysis for the linearized operator (Proposition \ref{prop-spec-sec}) and the nonlinear estimate \eqref{eq-nonlinear-est}.

In the following discussion, we fix a frequency $\om\in\R_+$ and suppress the dependence of $\psi(t;\om,A)$ on $\om$ in \eqref{eq-periodic-forcing-0} as
\EQ{\label{eq-periodic-forcing}
p_\infty(t) = p_{\infty,*} + \psi(t;A).
}
Here, $\psi(t;A)$ is $2\pi/\om$- periodic and bounded in $t$ such that $\psi(t;A)\to0$ as the amplitude $A\to0$.


\subsection{A dynamical system formulation of the periodically forced problem}\label{sec-dyn-sys-derivation}

To study the solution near the equilibrium of the unforced problem, we derive two equivalent systems of \eqref{red-eqns-mass} below.

\begin{proposition}\label{prop-equivalent-systems}
Given $(\rho_*,R_*)$ satisfying $(4\pi/3)\rho_*R_*^3 = M$ and $\Rg T_\infty\rho_* = p_{\infty,*} + 2\si/R_*$.
Let $(\rho,R)$ denote a solution of the nonlinear free boundary problem \eqref{red-eqns-mass} with $2\pi/\omega$- time periodic pressure field at infinity, \eqref{eq-periodic-forcing} with amplitude $A>0$.
Decompose 
\EQ{\label{eq-perturbation}
\rho(R(t)y,t) = \rho_* + u(y,t) + z(t),\qquad 
z(t) = \rho(R(t),t) - \rho_*,\qquad
R(t) = R_* + \mathcal R(t).
}
Then,
\begin{enumerate}
\item  \eqref{red-eqns-mass} is equivalent to the following system for $(u, \mathcal R)$ with zero-Dirichlet boundary condition 
\begin{subequations}
\label{red-eqns-linear-new-1}
\begin{align}
\pd_tu &= \bar\ka \De_yu - \bke{ 1 - \frac1{\ga} } \dot z + F,\quad 0\le y\le1,\qquad u(1,t)=0,\quad \ t>0,\label{eq-u-pde-1}\\
z(t) &= \frac1{\Rg T_\infty} \bke{ -\frac{2\si}{R_*^2}\mathcal R + \frac{4\mu_l}{R_*}\, \dot{\mathcal R} + \rho_lR_*\ddot{\mathcal R} } + H + \frac{\psi(t;A)}{\Rg T_\infty},\quad t>0,\label{eq-z-1}\\
\int_{B_1} u &= -\frac{4\pi}3 z - 4\pi \frac{\rho_*}{R_*}\, \mathcal R + G,\label{eq-conserv-mass-1}
\end{align}
\end{subequations}
where $F$ and $H$ are defined as in \cite[(9.4a), (9.4c)]{LW-vbas2022}, and
\EQ{\label{eq-def-G}
G = -\bke{ \frac{3\mathcal R^2}{R_*^2} + \frac{\mathcal R^3}{R_*^3} } \frac{4\pi}3\rho_* - \bke{\frac{3\mathcal R}{R_*} + \frac{3\mathcal R^2}{R_*^2} + \frac{\mathcal R^3}{R_*^3} } \bke{\int_{B_1} u + \frac{4\pi}3 z}.
}

\item Let $c_j = c_j(t)$ denote the $j$-th coefficients in the radial-Dirichlet-eigenfunction decomposition of $u$ as in \eqref{eq-u-expansion}.

Then, \eqref{red-eqns-linear-new-1} is further equivalent to the following infinite-dimensional dynamical system for ${\bf w}= (\mathcal R,\dot{\mathcal R}, c_1,c_2,\cdots)^\top$
\EQ{\label{eq-carr-6.3.1-simple-form} 
\dot{\bf w} = \mathcal L_- {\bf w} + \mathcal N^1({\bf w}) \dot{\bf w} + \mathcal N^0({\bf w}) - \frac{\psi(t;A)}{\rho_lR_*}
\begin{bmatrix}
0\\
1\\
\frac{3\rho_*}{R_*}\ga\Ga_1\\
\frac{3\rho_*}{R_*}\ga\Ga_2\\
\vdots
\end{bmatrix},
}
where $\mathcal L_-$ is a linear operator given in \eqref{eq-def-L-}, and $\mathcal N^1({\bf w})$ and $\mathcal N^0({\bf w})$ are given in \eqref{def-N1-N0-simple-form}.
\end{enumerate}
\end{proposition}

\begin{proof}
To begin with, plugging \eqref{eq-perturbation} into the equation \eqref{red-eqns-mass} and grouping linear and nonlinear terms, one directly obtain \eqref{red-eqns-linear-new-1}.
Note that \eqref{eq-conserv-mass-1} is equivalent to 
\EQ{\label{eq-conserv-mass-0}
z(t) &= \frac{\rho_* R_*^3}{(R_*+\mathcal R)^3} - \rho_* - \frac3{4\pi} \int_{B_1} u.
}

Using the Dirichlet eigenfunction expansion \eqref{eq-u-expansion} in \eqref{eq-u-pde-1} and \eqref{eq-conserv-mass-1} yields
\EQ{\label{eq-dotc-forced}
\dot c_k = -\bar\ka\la_kc_k - \Ga_k\dot z + F_k,\quad F_k = \int_{B_1} F\phi_k,
}
\EQ{\label{eq-z-reduction-1}
z = -\frac3{4\pi} \frac{\ga}{\ga-1} \sum_{j=1}^\infty \Ga_jc_j - \frac{3\rho_*}{R_*} \mathcal R + \frac3{4\pi} G.
}
Using \eqref{eq-z-reduction-1} in \eqref{eq-dotc-forced} and \eqref{eq-z-1}, we deduce
\EQ{\label{eq-dotc-forced-1}
\dot c_k = -\bar\ka\la_kc_k - \Ga_k \bke{ -\frac{3\ga}{4\pi(\ga-1)} \sum_{j=1}^\infty \Ga_j \dot c_j - \frac{3\rho_*}{R_*} \dot{\mathcal R} + \frac3{4\pi} \dot G } + F_k,
}
\EQ{\label{eq-z-2}
\ddot{\mathcal R} &= \frac{\Rg T_\infty}{\rho_lR_*} \bke{  -\frac{3\ga}{4\pi(\ga-1)} \sum_{j=1}^\infty \Ga_j c_j - \frac{3\rho_*}{R_*} \mathcal R + \frac3{4\pi} G } + \frac{2\si}{\rho_lR_*^3}\mathcal R - \frac{4\mu_l}{\rho_lR_*^2}\dot{\mathcal R} - \frac{\Rg T_\infty}{\rho_lR_*}H - \frac{\psi(t;A)}{\rho_lR_*}\\
&= - b \mathcal R - \frac{4\mu_l}{\rho_lR_*^2} \dot{\mathcal R} - d \sum_{j=1}^\infty \Ga_jc_j + \frac{3\Rg T_\infty}{4\pi\rho_lR_*} G - \frac{\Rg T_\infty}{\rho_lR_*} H - \frac{\psi(t;A)}{\rho_lR_*},
}
where $b$ and $d$ are defined as in \eqref{eq-def-b-d}.

Thus, \eqref{eq-z-2} and \eqref{eq-dotc-forced-1} form the infinite-dimensional dynamical system for ${\bf w} = (\mathcal R,\dot{\mathcal R}, c_1,c_2,\cdots)^\top$:
\EQN{
&\begin{pmatrix}
1&0&0&0&\cdots\\
0&1&0&0&\cdots\\
0&-\frac{3\rho_*}{R_*} \Ga_1&1-\frac{3\ga}{4\pi(\ga-1)}\Ga_1^2&-\frac{3\ga}{4\pi(\ga-1)}\Ga_1\Ga_2&\cdots\\
0&-\frac{3\rho_*}{R_*} \Ga_2&-\frac{3\ga}{4\pi(\ga-1)}\Ga_1\Ga_2&1-\frac{3\ga}{4\pi(\ga-1)}\Ga_2^2&\cdots\\
\vdots&\vdots&\vdots&\vdots&\ddots
\end{pmatrix}
\begin{bmatrix}
\mathcal R\\
\dot{\mathcal R}\\
c_1\\
c_2\\
\vdots
\end{bmatrix}^\prime\\
&\qquad=
\begin{pmatrix}
0&1&0&0&\cdots\\
-b & -\frac{4\mu_l}{\rho_lR_*^2}&-d\Ga_1&-d\Ga_2&\cdots\\
0&0&-\bar\ka\la_1&0&\cdots\\
0&0&0&-\bar\ka\la_2&\cdots\\
\vdots&\vdots&\vdots&\vdots&\ddots
\end{pmatrix}
\begin{bmatrix}
\mathcal R\\
\dot{\mathcal R}\\
c_1\\
c_2\\
\vdots
\end{bmatrix}
 + 
\begin{bmatrix}
0\\
\frac{3\Rg T_\infty}{4\pi\rho_lR_*} G - \frac{\Rg T_\infty}{\rho_lR_*}H\\
-\frac3{4\pi}\Ga_1\dot G + F_1\\
-\frac3{4\pi}\Ga_2\dot G + F_2\\
\vdots
\end{bmatrix}
 -\frac{\psi(t;A)}{\rho_lR_*}
\begin{bmatrix}
0\\
1\\
0\\
0\\
\vdots
\end{bmatrix},
}
where the inverse of the matrix on the left hand side above is derived in \eqref{eq-inverse-matrix}.
Left-multiplying the inverse on both sides, we obtain
\EQ{\label{eq-matrix-form}
\begin{bmatrix}
\mathcal R\\
\dot{\mathcal R}\\
c_1\\
c_2\\
\vdots
\end{bmatrix}^\prime
&=
\begin{pmatrix}
0&1&0&0&\cdots\\
-b &-\frac{4\mu_l}{\rho_lR_*^2}&-d\Ga_1&-d\Ga_2&\cdots\\
-b \frac{3\rho_*}{R_*}\ga\Ga_1 & -\frac{12\mu_l\rho_*}{\rho_lR_*^3}\ga\Ga_1 &-\bar\ka\la_1 - e_1\Ga_1^2& -e_2\Ga_1\Ga_2&\cdots\\
-b \frac{3\rho_*}{R_*}\ga\Ga_2 & -\frac{12\mu_l\rho_*}{\rho_lR_*^3}\ga\Ga_2 &- e_1\Ga_1\Ga_2&-\bar\ka\la_2 -e_2\Ga_2^2&\cdots\\
\vdots&\vdots&\vdots&\vdots&\ddots
\end{pmatrix}
\begin{bmatrix}
\mathcal R\\
\dot{\mathcal R}\\
c_1\\
c_2\\
\vdots
\end{bmatrix}\\
&\quad +
\begin{bmatrix}
0\\
\frac{3\Rg T_\infty}{4\pi\rho_lR_*} G - \frac{\Rg T_\infty}{\rho_lR_*}H\\
\frac{3\rho_*}{R_*}\ga\Ga_1\bke{ \frac{3\Rg T_\infty}{4\pi\rho_lR_*} G - \frac{\Rg T_\infty}{\rho_lR_*}H } - \frac3{4\pi} \Ga_1\dot G + F_1 + \frac{3\ga^2}{4\pi(\ga-1)}\Ga_1 \underset{j=1}{\overset{\infty}\sum} \Ga_j\bke{-\frac3{4\pi}\Ga_j\dot G + F_j}\\
\frac{3\rho_*}{R_*}\ga\Ga_2\bke{ \frac{3\Rg T_\infty}{4\pi\rho_lR_*} G - \frac{\Rg T_\infty}{\rho_lR_*}H } - \frac3{4\pi} \Ga_2\dot G + F_2 + \frac{3\ga^2}{4\pi(\ga-1)}\Ga_2 \underset{j=1}{\overset{\infty}\sum} \Ga_j\bke{-\frac3{4\pi}\Ga_j\dot G + F_j}\\
\vdots
\end{bmatrix}\\
&\quad -\frac{\psi(t;A)}{\rho_lR_*}
\begin{bmatrix}
0\\
1\\
\frac{3\rho_*}{R_*}\ga\Ga_1\\
\frac{3\rho_*}{R_*}\ga\Ga_2\\
\vdots
\end{bmatrix},
}
where $e_j$ is defined as in \eqref{eq-def-ej}.

For the nonlinear term, using \eqref{eq-u-expansion} in \eqref{eq-conserv-mass-0},
\EQ{\label{eq-z-reduction-0}
z = \frac{\rho_*R_*^3}{(R_*+\mathcal R)^3} - \rho_* - \frac3{4\pi} \frac{\ga}{\ga-1} \sum_{j=1}^\infty \Ga_j c_j
= z(\mathcal R, c_1, c_2,\cdots).
}
Plugging \eqref{eq-z-reduction-0} into the expression of $F$ in \cite[(9.4a)]{LW-vbas2022}, we obtain
for ${\bf w} = (\mathcal R,\dot{\mathcal R}, c_1,c_2,\cdots)^\top$, ${\bf p} = (\dot{\mathcal R},\ddot{\mathcal R}, \dot c_1, \dot c_2,\cdots)^\top =: (\mathcal S, \mathcal U, d_1,d_2,\cdots)^\top$ that
\EQN{
F({\bf w}, {\bf p}) &= \frac{\ka}{\ga c_v} \bkt{\frac1{(R_*+\mathcal R)^2 \bke{\rho_*+ u +z} } - \frac1{R_*^2\rho_*}} \sum_{j=1}^\infty (-\la_j)c_j\phi_j - \frac{\ka}{\ga c_v}\, \frac{ \abs{\nb_y u}^2}{(R_*+\mathcal R)^2 \bke{\rho_*+ u +z}^2}\\
&\quad - \frac3{4\pi(\ga-1)(\rho_* + z)} \sum_{j=1}^\infty \Ga_j\dot c_j \bke{\frac13 y \sum_{j=1}^\infty c_j\pd_y\phi_j + \sum_{j=1}^\infty c_j\phi_j }\\
&= F^0({\bf w}) + {\bf F}^1({\bf w}){\bf p},
}
where $F^0$ is the same as in \cite[(9.16)]{LW-vbas2022} and
\[
{\bf F}^1({\bf w}) = -\frac3{4\pi(\ga-1)(\rho_*+z)} \bke{\frac13 y \sum_{j=1}^\infty c_j\pd_y\phi_j + \sum_{j=1}^\infty c_j\phi_j } (0,0,\Ga_1,\Ga_2,\cdots).
\]
For $G$, using \eqref{eq-conserv-mass-1} in \eqref{eq-def-G},
\[
G = -\bke{ \frac{3\mathcal R^2}{R_*^2} + \frac{\mathcal R^3}{R_*^3} } \frac{4\pi}3\rho_* - \bke{\frac{3\mathcal R}{R_*} + \frac{3\mathcal R^2}{R_*^2} + \frac{\mathcal R^3}{R_*^3} } \bke{ \frac{4\pi R_*^3\rho_*}{(R_*+\mathcal R)^3} - \frac{4\pi}3 \rho_*} = G(\mathcal R).
\] 
Thus, $\dot G = G'(\mathcal R)\dot{\mathcal R}$.
In particular, $G = G_0({\bf w})$ and $\dot G = G_1({\bf w})$ for some functions $G_0$ and $G_1$.

Therefore, \eqref{eq-matrix-form} can be expressed as the periodically forced dynamical system:
\[
\dot{\bf w} = \mathcal L_- {\bf w} + \mathcal N^1({\bf w}) \dot{\bf w} + \mathcal N^0({\bf w}) - \frac{\psi(t;A)}{\rho_lR_*} 
\begin{bmatrix}
0\\
1\\
\frac{3\rho_*}{R_*}\ga\Ga_1\\
\frac{3\rho_*}{R_*}\ga\Ga_2\\
\vdots
\end{bmatrix}.
\]
Here, the linearized operator $\mathcal L_-$ is displayed in \eqref{eq-def-L-}, and the nonlinear terms are given by:
\EQ{\label{def-N1-N0-simple-form}
\mathcal N^1({\bf w}) &= 
 \begin{pmatrix}
0\\
-\frac{\Rg T_\infty}{\rho_lR_*}\, {\bf H}^1({\bf w})\\
-\frac{3\rho_*}{R_*}\ga\Ga_1 \frac{\Rg T_\infty}{\rho_lR_*}{\bf H}^1({\bf w}) + {\bf F}^1_1({\bf w}) + \frac{3\ga^2}{4\pi(\ga-1)}\Ga_1\sum_{j=1}^\infty {\bf F}^1_j({\bf w})\\
-\frac{3\rho_*}{R_*}\ga\Ga_2 \frac{\Rg T_\infty}{\rho_lR_*}{\bf H}^1({\bf w}) + {\bf F}^1_2({\bf w}) + \frac{3\ga^2}{4\pi(\ga-1)}\Ga_2\sum_{j=1}^\infty {\bf F}^1_j({\bf w})\\
\vdots
 \end{pmatrix},\\
\mathcal N^0({\bf w}) &=  
{\tiny
 \begin{bmatrix}
0\\
\frac{3\Rg T_\infty}{4\pi\rho_lR_*} G_0({\bf w}) - \frac{\Rg T_\infty}{\rho_lR_*}H^0({\bf w})\\
\frac{3\rho_*}{R_*}\ga\Ga_1\bke{ \frac{3\Rg T_\infty}{4\pi\rho_lR_*} G_0({\bf w}) - \frac{\Rg T_\infty}{\rho_lR_*}H^0({\bf w}) } - \frac3{4\pi} \Ga_1 G_1({\bf w}) + F_1^0({\bf w}) + \frac{3\ga^2}{4\pi(\ga-1)}\Ga_1 \underset{j=1}{\overset{\infty}\sum} \Ga_j \bke{-\frac3{4\pi}\Ga_j G_1({\bf w}) + F_j^0({\bf w})}\\
\frac{3\rho_*}{R_*}\ga\Ga_2\bke{ \frac{3\Rg T_\infty}{4\pi\rho_lR_*} G_0({\bf w}) - \frac{\Rg T_\infty}{\rho_lR_*}H^0({\bf w}) } - \frac3{4\pi} \Ga_2 G_1({\bf w}) + F_1^0({\bf w}) + \frac{3\ga^2}{4\pi(\ga-1)}\Ga_2 \underset{j=1}{\overset{\infty}\sum} \Ga_j \bke{-\frac3{4\pi}\Ga_j G_1({\bf w}) + F_j^0({\bf w})}\\
\vdots
 \end{bmatrix},
 }
}
where 
\begin{align*}
{\bf F}_j^1({\bf w}) &= \int_{B_1} {\bf F}^1({\bf w}) \phi_j\, dx,\quad  F_j^0({\bf w}) = \int_{B_1} F^0({\bf w}) \phi_j\, dx,\\
 G_0({\bf w}) &= G,\quad  G_1({\bf w}) = \dot G,\\
{\bf H}^1({\bf w}) &= (0,\frac{\rho_l\mathcal R}{\Rg T_\infty},0,0,\cdots),
\end{align*} and
\[
H^0({\bf w}) = \frac1{\Rg T_\infty} \bkt{ -\frac{\mathcal R}{R_*(R_*+\mathcal R )}\bke{ -\frac{2\si}{R_*}\, \mathcal R + 4\mu_l \dot{\mathcal R} }  + \rho_l \frac32 \dot{\mathcal R}^2 }.
\]
This completes the proof of Proposition \ref{prop-equivalent-systems}.
\end{proof}

\subsection{Global well-posedness of the periodically forced problem near the equilibrium of the unforced model}\label{sec-global-wellp}

In the following, we prove the  existence and uniqueness  of global-in-time solutions to the initial value problem for the periodically forced problem. Our initial data are taken to be $C^{2+2\al}-$ close to a spherical equilibrium of the unforced model. These results on existence and regularity enable us to construct
 the  Poincar\'e map and prove the existence of its fixed points for small amplitude forcing.
Hereafter, for any $(\rho(r,t),R(t))$ defined for $0\le r\le R(t)$, $t>0$, we denote
\[
\overline\rho(y,t) = \rho(R(t)y,t)\ \text{ for }0\le y\le 1.
\]

\begin{theorem}\label{thm-period-wellposed}
Let $\ve_0>0$ be arbitrary.
There exist constants $A_0>0$ and $\eta_0>0$ such that if the initial data $\rho_0(r), R_0, \dot R_0$ satisfies
\EQ{\label{data-small-0}
\oldnorm{(\overline\rho_0(\cdot) - \rho_*,R_0 - R_*,\dot R_0) }_{C^{2+2\al}} \le \eta_0,
}
where $\overline\rho_0(y) = \rho_0(R_0y)$, $\oldnorm{\,\cdot\,}_{C^{2+2\al}}$ is defined in \eqref{eq-def-oldnorm-holder},
and
$(\rho_*,R_*)$ is the equilibrium of the unforced problem with the same mass, $(4\pi/3) R_*^3\rho_* = \int_{B_{R_0}} \rho_0\, dx$ and $\Rg T_\infty\rho_* = p_{\infty,*} + 2\si/R_*$,
then for every fixed $A\in(0,A_0)$
there exists a unique solution $(\rho,R)$ to the nonlinear free boundary problem \eqref{red-eqns} with $2\pi/\omega$- time periodic pressure field at infinity, \eqref{eq-periodic-forcing} with amplitude $A>0$, satisfying 
\EQ{\label{eq-bv-4.43}
\oldnorm{(\overline\rho(\cdot,t) - \rho_*,R(t) - R_*,\dot R_0(t)) }_{C^{2+2\al}} \le \ve_0,\qquad \text{ for all }t>0,
}
where $\overline\rho(y,t) = \rho(R(t)y,t)$.
\end{theorem}

\begin{proof}
First, note that the  existence of unique local-in-time solutions of \eqref{red-eqns} is established in \cite[Theorem 3.1]{BV-SIMA2000} without the smallness assumption on initial data.
To prove the solution exists globally in time, it suffices to prove the estimate \eqref{eq-bv-4.43}.
Certainly, \eqref{eq-bv-4.43} holds for short time if $\eta_0>0$ in \eqref{data-small-0} is sufficiently small
by the local-in-time wellposedness given by \cite[Theorem 3.1]{BV-SIMA2000}.
To prove \eqref{eq-bv-4.43} holds for all time, we look at the equivalent system \eqref{eq-carr-6.3.1-simple-form} in Proposition \ref{prop-equivalent-systems}. 
For sufficiently small ${\bf w} = (\mathcal R,\dot{\mathcal R}, c_1,c_2,\cdots)^\top$, one can rewrite \eqref{eq-carr-6.3.1-simple-form} as
\EQ{\label{eq-dyn-sys-1}
\dot{\bf w} = \mathcal L_-{\bf w} + f(t,{\bf w}, A),
}
where
\[
f(t,{\bf w}, A) := \bkt{(I-\mathcal N^1({\bf w}))^{-1} \mathcal L_- - \mathcal L_-}{\bf w} + (I-\mathcal N^1({\bf w}))^{-1} 
\bket{\mathcal N^0({\bf w}) - \frac{\psi(t;A)}{\rho_lR_*} 
\begin{bmatrix}
0\\
1\\
\frac{3\rho_*}{R_*}\ga\Ga_1\\
\frac{3\rho_*}{R_*}\ga\Ga_2\\
\vdots
\end{bmatrix}
}.
\]
Since $\rho(\cdot,t)\in C^{2+2\al}_r$ for short time, we have $\{j^2c_j(t)\}_{j=1}^\infty\in\ell^2$, and thus $f(t,{\bf w}(t),A)\in\ell^2$ for short time.

To derive an a priori estimate of the solution ${\bf w}$ in $\ell^2$, we note that 
$\mathcal N^1({\bf w}) = O(\norm{\bf w})$ and $\mathcal N^0({\bf w}) = O(\norm{\bf w}^2)$ so that
\EQ{\label{eq-nonlinear-est}
f(t,{\bf 0},A) = -\frac{\psi(t;A)}{\rho_lR_*} \begin{bmatrix}
0\\
1\\
\frac{3\rho_*}{R_*}\ga\Ga_1\\
\frac{3\rho_*}{R_*}\ga\Ga_2\\
\vdots
\end{bmatrix}
\quad \text{ and }\quad
\pd_{\bf w}f(t,{\bf 0},A)=\pd_{\bf w}\mathcal N^1({\bf 0}) \bket{ -\frac{\psi(t;A)}{\rho_lR_*} \begin{bmatrix}
0\\
1\\
\frac{3\rho_*}{R_*}\ga\Ga_1\\
\frac{3\rho_*}{R_*}\ga\Ga_2\\
\vdots
\end{bmatrix} }.
}
By Duhamel's principle, 
\[
{\bf w}(t) = e^{\mathcal L_-} {\bf w}(0) + \int_0^t e^{\mathcal L_-(t-s)} f(s,{\bf w}(s),A)\, ds.
\] 
Thus, by \eqref{eq-semigroup-decay-est}
\EQN{
\norm{{\bf w}(t)} 
&\le Ce^{-\be t} \norm{{\bf w}(0)} + \int_0^t Ce^{-\be(t-s)} \norm{f(s,{\bf w}(s),A)} ds,
}
where $\be>0$ is given in \eqref{eq-be-sharp-def}.
By \eqref{eq-nonlinear-est}, we deduce that if $\norm{\bf w}<\ve$
\EQN{
\norm{{\bf w}(t)} 
&\le C\norm{{\bf w}(0)} + \int_0^t Ce^{-\be(t-s)} \norm{f(s,{\bf w}(s),A) - f(s,{\bf 0},A)} ds
+ \int_0^t Ce^{-\be(t-s)}\norm{f(s,{\bf 0},A)} ds\\
&\le C\norm{{\bf w}(0)} + C\de(A)\ve \int_0^t e^{-\be(t-s)}\, ds + ACC_1 \int_0^t e^{-\be(t-s)}\, ds,
}
for some $C_1>0$ and continuous function $\de(A)$ with $\de(0)=0$.
Therefore, we obtain for some $C_2>0$ that
\EQ{\label{eq-w-global-small}
\norm{{\bf w}(t)} \le C\norm{{\bf w}(0)} + C\de(A)\ve C_2 + ACC_1C_2
}
which can be made arbitrarily small if $\norm{{\bf w}(0)}$ and $A$ are so.
This amounts to uniformly smallness of $\norm{\rho(\cdot,t) - \rho_*}_{L^2} + |R(t) - R_*| + |\dot R(t)|$ given that $\eta_0>0$ in \eqref{data-small-0} is sufficiently small.
With the smallness of $\norm{\rho(\cdot,t) - \rho_*}_{L^2}$, one can
use the interpolation lemma \cite[Lemma D.1]{LW-vbas2022} and the bootstrap argument performed in the proof of \cite[Theorem 4.1]{BV-SIMA2000} to conclude that \eqref{eq-bv-4.43} holds for all $t$ and complete the proof of Theorem \ref{thm-period-wellposed}.

\end{proof}

\subsection{Existence, uniqueness, and stability of the time-periodic solution}

Below, by constructing a Poincar\'e return map, we successfully prove the uniquely existence and stability of the time periodic solutions of the nonlinear bubble-fluid system  for the small-amplitude periodic forcing.

\begin{theorem}\label{cor-period-exituniq}
Consider the nonlinear free boundary problem \eqref{red-eqns} with $2\pi/\omega$- time periodic pressure field at infinity, \eqref{eq-periodic-forcing}.
\begin{enumerate}
\item Fix $M>0$ an arbitrary bubble mass.
 There exists $A_0(M)$ sufficiently small and positive such that
the system \eqref{red-eqns} has a unique $2\pi/\om$-periodic solution:
\[ \rho_{\rm per}[M,A](\cdot,t), R_{\rm per}[M,A](t),\quad t\in\R_+,\]
of bubble mass $M$, 
which depends smoothly on $M$ and $A$ for $0<M$ and $0<A<A_0(M)$.

\item The $2\pi/\om$-periodic solution $\rho_{\rm per}[M,A](t), R_{\rm per}[M,A](t)$ is (nonlinearly) exponentially stable with respect to mass preserving perturbations.
That is, there exists a constant $\eta>0$ such that if
\EQN{
\oldnorm{ \bke{\overline\rho_0(\cdot) - \overline\rho_{\rm per}[M,A](\cdot,0), R_0 - R_{\rm per}[M,A](0), \dot R_0 - \dot R_{\rm per}[M,A](0) } }_{L^2} \le \eta,\quad
\int_{B_{R_0}} \rho_0\, dx = M,
}
where 
$\overline\rho_0(y) = \rho_0(R_0y)$, $\overline\rho_{\rm per}(y,t) = \rho_{\rm per}(R_{\rm per}(t)y,t)$, 
$\oldnorm{\,\cdot\,}_{L^2}$ is defined in \eqref{eq-def-oldnorm},
and if $(\rho(\cdot,t), R(t), \dot R(t))$ is the global solution of \eqref{red-eqns} with initial data $(\rho_0, R_0, \dot R_0)\in C^{2+2\al}_r\times\R_+\times\R$, 
then
\EQN{
&\oldnorm{ \bke{\overline\rho(\cdot,t) - \overline\rho_{\rm per}[M,A](\cdot,t), R(t) - R_{\rm per}[M,A](t), \dot R(t) - \dot R_{\rm per}[M,A](t) } }_{L^2}\\
&\qquad\qquad\qquad\qquad\qquad\qquad\qquad\qquad\qquad\qquad\qquad\qquad\ \, = O\bke{e^{-(\be + o_A(1))}} \ \text{ as }\ t\to+\infty,
}
where 
$\overline\rho(y,t) = \rho(R(t)y,t)$, 
$\be>0$ is given in \eqref{eq-be-sharp-def} and $o_A(1)\to0$ as $A\to0$.

\item The manifold of periodic solutions:
 \EQ{\label{eq-manifold-equilib-per}
\mathcal M_{\rm per} = \bket{ (\rho_{\rm per}[M,A](t), R_{\rm per}[M,A](t), \dot R_{\rm per}[M,A](t)) : 0\le t\le 2\pi/\omega,\ 0<M<\infty,\ 0<A<A_0(M) },
} is (nonlinearly) asymptotically stable.
More precisely, there exists a constant $\eta_1>0$ such that
for any initial data $(\rho_0, R_0, \dot R_0)\in C^{2+2\al}_r\times\R_+\times\R$ with ${\rm dist}((\rho_0, R_0, \dot R_0), \mathcal M_{\rm per}) \le \eta_1$ we have
\[
{\rm dist}((\rho(\cdot,t), R(t), \dot R(t)), \mathcal M_{\rm per}) 
= O\bke{e^{-(\be + o_A(1))}} 
\ \text{ as }\ t\to+\infty,
\]
where 
\EQN{
{\rm dist}&( (\rho,R,\dot R), \mathcal{M}_{\rm per})\\
&\equiv \inf \bket{ \oldnorm{\left(\overline\rho - \overline\rho_{\rm per}, R-R_{\rm per}, \dot R-\dot R_{\rm per}\right)}_{L^2} : (\rho_{\rm per},R_{\rm per},\dot R_{\rm per})\in \mathcal{M}_{\rm per}}\\
&=\inf_{0<M<\infty}
\oldnorm{ \left(\overline\rho-\overline\rho_{\rm per}[M,A], R-R_{\rm per}[M,A], \dot R - \dot R_{\rm per}[M,A]\right)}_{L^2}.
}
\end{enumerate}
\end{theorem}

\begin{remark}
In the same spirit as in Remark \ref{rmk-higher-norm-conv}, 
with the aid of the uniform boundedness \eqref{eq-bv-4.43} in $C^{2+2\al}$ (see Theorem \ref{thm-period-wellposed}) 
one can bootstrap the $L^2$-exponential decay in (2) and (3) of Theorem \ref{cor-period-exituniq} to $C^{2+2\al}$-exponential decay with a cost of decay exponent.
More precisely, one has
\EQ{
&\norm{\overline\rho(\cdot,t) - \overline\rho_{\rm per}(t)}_{C^{2+2\al}_y(B_1)},\ 
|\ddot R(t) - \ddot R_{\rm per}(t)|,\ 
|\dddot R(t) - \dddot R_{\rm per}(t)|
 = O\bke{ e^{-\de(\be + o_A(1))t} }\ \text{ as }\ t\to\infty,
}
for some $\de<1$, where $\be>0$ is as in \eqref{eq-be-sharp-def}.
\end{remark}

\begin{remark}
All the conclusions on $(\rho(x,t),R(t))$ in Theorem \ref{cor-period-exituniq} transfer --via the explicit expressions relating gas to fluid variables in \cite[Proposition 5.2]{LW-vbas2022}-- to statements about the full bubble-fluid system \cite[(3.1)-(3.4)]{LW-vbas2022}, with $p_\infty(t) = p_{\infty,*} + \psi(t;A)$ in \cite[(3.4)]{LW-vbas2022}, on other state variables.
In particular, the system \cite[(3.1)-(3.4)]{LW-vbas2022} admits a unique exponentially stable time-periodic solution $({\bf v}_{l,{\rm per}}(x,t), p_{l,{\rm per}}(x,t), B_{R_{\rm per}(t)}, p_{g,{\rm per}}(t), {\bf v}_{g,{\rm per}}(x,t), \rho_{g,{\rm per}}(x,t),\cdots)$ respect to small mass-preserving spherically symmetric perturbations.
\end{remark}

\begin{remark}
A more detailed picture of analysis near the manifold of periodic solutions $\mathcal M_{\rm per}$ can be obtained by a similar center manifold analysis as in \cite[Section 9]{LW-vbas2022}. 
\end{remark}

\begin{proof}[Proof of Theorem \ref{cor-period-exituniq}]
For a fixed bubble mass $M>0$, we consider the equivalent system \eqref{red-eqns-mass} with initial bubble mass $M$, $\int_{B_{R(0)}}\rho(x,0)\,dx = M$.
The  existence of a unique global-in-time solution to the initial value problem for \eqref{red-eqns-mass} in the $C^{2+2\al}$ framework is derived in Theorem \ref{thm-period-wellposed}.
Note that \eqref{red-eqns-mass} is equivalent to the infinite-dimensional dynamical system \eqref{eq-carr-6.3.1-simple-form} by Proposition \ref{prop-equivalent-systems}. 
We write it in the form:
\EQ{\label{eq-dyn-sys-1-same}
\dot{\bf w} = \mathcal L_- {\bf w} + f(t,{\bf w},A);
}
see  \eqref{eq-dyn-sys-1}.
Denote ${\bf w}(t;{\bf w}_0,A)$ the unique solution of \eqref{eq-dyn-sys-1-same} with initial data ${\bf w}(0;{\bf w}_0,A)={\bf w}_0$.
Since $f(t,{\bf w},A)$ is $2\pi/\om$ periodic in $t$, the mapping $t\mapsto {\bf w}(t;{\bf w}_0,A)$ is $2\pi/\om-$ periodic if  and only if  ${\bf w}_0$ is a fixed point of the Poincar\'e map ${\bf w}_0\mapsto {\bf w}(2\pi/\om;{\bf w}_0,A)$.
By Duhamel's principle, 
\[
{\bf w}(t;{\bf w}_0,A) = e^{\mathcal L_- t} {\bf w}_0 + \int_0^t e^{\mathcal L_-(t-s)} f(s,{\bf w}(s;{\bf w}_0,A),A)\, ds.
\]
Hence, ${\bf w}_0 = {\bf w}(2\pi/\om;{\bf w}_0,A)$ if and only if
\begin{equation}\label{eq:FP}
{\bf w}_0 = \bke{ I - e^{\mathcal L_- \frac{2\pi}\om }}^{-1} \int_0^{\frac{2\pi}\om} e^{\mathcal L_-\bke{\frac{2\pi}\om - s}} f(s,{\bf w}(s;{\bf w}_0,A),A)\, ds.
\end{equation}
We solve \eqref{eq:FP} for $ {\bf w}_0 = {\bf g}(A)$ by applying the implicit function theorem to the mapping:
\[
\Phi({\bf w}_0,A) = {\bf w}_0 - \bke{ I - e^{\mathcal L_- \frac{2\pi}\om }}^{-1} \int_0^{\frac{2\pi}\om} e^{\mathcal L_-\bke{\frac{2\pi}\om - s}} f(s,{\bf w}(s;{\bf w}_0,A),A)\, ds.
\]
It is easy to see that $\Phi({\bf 0},0) = {\bf 0}$ since $f(t,{\bf 0},0) = {\bf 0}$ by \eqref{eq-nonlinear-est}.
We now compute $\pd_{{\bf w}_0}\Phi({\bf 0},0)$.
Since that ${\bf w}(t,{\bf 0},0) = {\bf 0}$ and $\pd_{\bf w}f(t,{\bf 0},0) = {\bf O}$ by \eqref{eq-nonlinear-est},
and since $e^{\mathcal L_- t}$ is an analytic semi-group (Proposition \ref{prop-spec-sec}), we have
\[
\pd_{{\bf w}_0}\Phi({\bf 0},0) = I - \bke{ I - e^{\mathcal L_- \frac{2\pi}\om }}^{-1} \int_0^{\frac{2\pi}\om} e^{\mathcal L_-\bke{\frac{2\pi}\om - s}} \pd_{\bf w}f(s,{\bf w}(s;{\bf 0},0),0) \pd_{{\bf w}_0}{\bf w}(s;{\bf 0},0)\, ds = I,
\]
which is an isomorphism from $\ell^2$ to $\ell^2$.
Therefore, by the implicit function theorem there is a Fr\'echet differentiable function ${\bf g} = {\bf g}(A)$ and a neighborhood of $A=0$ in which $\Phi({\bf g}(A),A) = {\bf 0}$.
Let ${\bf w}_{\rm per}(t;A)$ be the solution of \eqref{eq-dyn-sys-1-same} with the initial data ${\bf w}_{\rm per}(0;A) = {\bf g}(A)$.
Then ${\bf w}_{\rm per}(t;A)$ is $2\pi/\om$- time periodic.
This proves (1) the existence of $2\pi/\om$- time periodic solution $(\rho_{\rm per}, R_{\rm per}, \dot R_{\rm per})$ of the free boundary problem \eqref{red-eqns} for a given bubble mass $M$.

We now prove the asymptotic stability of this solution. 
Let $(\rho, R)\in C^{2+2\al}_r\times\R_+$ be the unique solution of \eqref{red-eqns} with the initial data $(\rho_0, R_0, \dot R_0)\in C^{2+2\al}_r\times\R_+\times\R$ chosen to be close to the data for the periodic solution $(\rho_{\rm per}, R_{\rm per}, \dot R_{\rm per})$.
Correspondingly, we have  ${\bf w}$ the solution of \eqref{eq-dyn-sys-1-same} with  initial data ${\bf w}_0$, (corresponding to $(\rho_0, R_0, \dot R_0)$) that is close to the data for the periodic solution ${\bf w}_{\rm per}$ 
(arising from the data $(\rho_{\rm per}, R_{\rm per}, \dot R_{\rm per})$).
Let ${\bf v} := {\bf w} - {\bf w}_{\rm per}$.
Then ${\bf v}$ satisfies
\[
\dot{\bf v} = \mathcal L_- {\bf v} + f(t,{\bf w}_{\rm per}+{\bf v},A) - f(t,{\bf w}_{\rm per},A).
\]
By Duhamel's principle,
\EQ{\label{eq-duhamel}
{\bf v}(t) = e^{\mathcal L_- t} {\bf v}_0 + \int_0^t e^{\mathcal L_-(t-s)} \bkt{ f(s,{\bf w}_{\rm per}(s)+{\bf v}(s),A) - f(s,{\bf w}_{\rm per}(s),A) } ds.
}
Note that
\[
\norm{ f(s,{\bf w}_{\rm per}(s)+{\bf v}(s),A) - f(s,{\bf w}_{\rm per}(s),A) } 
\le \norm{\pd_{\bf w} f(s, {\bf w}_{\rm per}(s), A)} \norm{\bf v(s)}.
\]
The nonlinear estimate \eqref{eq-nonlinear-est} implies $\norm{\pd_{\bf w} f(s, {\bf w}_{\rm per}(s), A)} \le \de({\bf w}_{\rm per}(s),A)$ for some continuous function $\de$ with $\de(0,0)=0$.
So 
\EQ{\label{eq-est-duhamel}
\norm{ f(s,{\bf w}_{\rm per}(s)+{\bf v}(s),A) - f(s,{\bf w}_{\rm per}(s),A) } 
\le \de({\bf w}_{\rm per}(s),A) \norm{\bf v(s)}.
}
Given that ${\bf g}(A) = {\bf w}_{\rm per}(0;A)$ is continuous in $A$ and that ${\bf w}_{\rm per}(0;0) = {\bf 0}$, we can make $\norm{{\bf w}_{\rm per}(0;A)}$ arbitrarily small by choosing sufficiently small $A$.
On the other hand, according to \eqref{eq-w-global-small} in the proof of Theorem \ref{thm-period-wellposed}, $\sup_{t\ge0}\norm{{\bf w}_{\rm per}(t;A)}$ can be made arbitrarily small if the data $\norm{{\bf w}_{\rm per}(0;A)}$ and the amplitude $A$ are small enough.
Thus, $\sup_{t\ge0}\norm{{\bf w}_{\rm per}(t;A)} = o_1(A)$ as $A\to0$.

Using the above result with \eqref{eq-est-duhamel} in the Duhamel's formula \eqref{eq-duhamel}, we have that
\EQN{\label{eq:v-est}
\norm{{\bf v}(t)} 
\le Ce^{-\be t} \norm{{\bf v}(0)} + C\de_1(A)\int_0^t e^{-\be(t-s)} \norm{{\bf v}(s)} ds,
}
for some continuous function $\de_1$ with $\de_1(0)=0$, where $\be>0$ is given in \eqref{eq-be-sharp-def}.
%
%
Applying Gronwall's lemma, we obtain
\[
\norm{{\bf v}(t)} \le C\norm{{\bf v}(0)} e^{-\bke{\be - C\de_1(A)} t}.
\]
Consequently, if $A$ is sufficiently small, ${\bf v}(t)$ decays exponentially, and thus ${\bf w}_{\rm per}(t;A)$ is exponentially asymptotically stable.
This amounts to the exponentially asymptotic stability of the time periodic solution $(\rho_{\rm per}[M,A], R_{\rm per}[M,A], \dot R_{\rm per}[M,A])\in \mathcal M_{\rm per}$ of the free boundary problem \eqref{red-eqns} with respect to mass-preserving perturbations, completing the proof of (2).

For a fixed sufficiently small amplitude $A$, the above result yields a continuous correspondence between the bubble mass $M$ and the initial data $(\rho_{\rm per}(\cdot,0), R_{\rm per}(0), \dot R_{\rm per}(0))$ of the unique periodic solution of \eqref{red-eqns}.
The stability relative to general (non mass-preserving) perturbations then follows from the continuous dependence of the solution of \eqref{red-eqns} on the initial data (See the proof in \cite[Theorem 3.1]{BV-SIMA2000}).
This proves (3) the asymptotic stability of the manifold $\mathcal M_{\rm per}$.
\end{proof}

\appendix

\section{
Estimates of the exponential rate, $\beta_{\rm thermal}$, and comparison with previous results}
\label{sec-compare-prosperetti}

%


%
Let $\beta_{\rm thermal}>0$ denote the thermal dissipation rate constant associated with the inviscid linearized dynamics of infinitesimal perturbations, {\it i.e.}  perturbation solutions
decay like $\exp(-\beta_{\rm thermal} t)$.
We compare our bounds with previous results for different regimes of the {\it thermal diffusivity parameter} \[ \chi = \frac{\kappa}{c_p\rho_*} = \bar\ka R_*^2>0.\]
In the inviscid case,  $\mu_l=0$, the decay rate \eqref{eq-be-sharp-def} in Theorem \ref{thm-exp-decay-LT} yields the following lower bound on the thermal dissipation rate
\EQ{\label{eq-be-thermal-est}
&\be_{\rm thermal} \ge \min\Bigg\{ \bke{1 - \sqrt{\dfrac{\vartheta(\ga)}{ \frac{p_{\infty,*}R_*}{2p_{\infty,*}R_*+6\si} + \vartheta(\ga) }}} \frac{\pi^2\chi}{R_*^2},\  \sqrt{ \e \frac{2p_*}{\rho_lR_*^2} },\\
&\qquad\qquad\qquad\qquad \frac{(1-\ve)^2\vartheta(\ga)p_*}{\pi^4\rho_l\chi} \bke{\frac{4\pi^4}{90} + (1-\ve)^{3/4} O\bke{\bke{\frac{R_*^2}{\pi^2\chi}\sqrt{ \frac{2p_*}{\rho_lR_*^2 }}}^{3/2}}} \Bigg\}.
}
Here, $\vartheta(\gamma)\equiv 1-\gamma^{-1}$ and the parameter $\ve\in(0,1)$ can be chosen arbitrarily.

The lower bound \eqref{eq-be-thermal-est} enables us to address two parameter regimes of interest:\\
(i) {\it the nearly isothermal regime}:  $\chi$ large, corresponding to rapid thermal diffusion, and\\
(ii) {\it the nearly adiabatic regime}:  $\chi$ small, corresponding to slow thermal diffusion.

The third expression under the minimum in \eqref{eq-be-thermal-est} dominates for $\chi$ large,
and the first expression dominates for $\chi$ small.

Specifically, fix $\ve\in(0,1)$. Then, for $\chi> \chi_0(\ve)$ sufficiently large, we have $\frac{4(1-\ve)^2\vartheta(\ga)p_*}{90\rho_l\chi} \le \sqrt{\e \frac{2p_*}{\rho_lR_*^2}}$.
Thus,
\eqref{eq-be-thermal-est} yields 
 \[
 \beta_{\rm thermal} >  \frac{4(1-\ve)^2\vartheta(\ga)p_*}{90\rho_l\chi}\qquad \text{ for all $\chi>\chi_0(\ve)$.}
 \]
Since $\ve\in(0,1)$ can be arbitrarily close to zero, we formally have 
 \[
 \beta_{\rm thermal} >  \frac{4\vartheta(\ga)p_*}{90\rho_l\chi}\qquad \text{ for large $\chi$.}
 \]
For $\chi$ small, \eqref{eq-be-thermal-est} implies
 \[
 \beta_{\rm thermal} \ge \bke{1 - \sqrt{\dfrac{\vartheta(\gamma)}{\frac{p_{\infty,*}R_*}{2p_{*,\infty}R_*+6\si} + \vartheta(\gamma)}}} \dfrac{\pi^2\chi}{R_*^2}.
 \]
This leads to expressions  for $\beta_{\rm isothermal}$ and $\beta_{\rm adiabatic}$, which are lower bounds
for the exponential decay rate in the isothermal, respectively adiabatic parameter regimes.
\EQ{\label{eq-thermal-damping-sum}
 \beta_{\rm thermal} >
 \begin{cases} 
 \dfrac{4 \vartheta(\gamma) p_*}{90\rho_l\chi} :=  \beta_{\rm isothermal}, \quad &\textrm{for $\chi$ large (isothermal)},\\
 \bke{1 - \sqrt{\dfrac{\vartheta(\gamma)}{\frac{p_{\infty,*}R_*}{2p_{*,\infty}R_*+6\si} + \vartheta(\gamma)}}} \dfrac{\pi^2\chi}{R_*^2} =: \beta_{\rm adiabatic},\quad &\textrm{for $\chi$ small (adiabatic)}.
 \end{cases}
 }
Note that \eqref{eq-thermal-damping-sum} gives lower bound estimates.

\subsection{Comparison of \eqref{eq-thermal-damping-sum} with Prosperetti \cite{Prosperetti-JFM1991}} 
 In the nearly isothermal case ($\chi$ large), Prosperetti deduced the following thermal dissipation rate via Laplace transforms, the linearized constant mass assumption on the initial data \cite[(3.15)]{Prosperetti-JFM1991} and formal approximations:
\EQ{\label{eq-thermal-damping-P91}
\be_{\rm thermal, P91} \approx \frac{\vartheta(\ga)p_*}{10\rho_l\chi},\qquad\text{ for $\chi$ large.}
}
The approximate rate $\be_{\rm thermal, P91}$ is consistent with our lower bound estimate \eqref{eq-thermal-damping-sum}.
In fact, it is $\be_{\rm thermal, P91}=9/4\times \be_{\rm isothermal}$, where $\be_{\rm isothermal}$ is displayed in \eqref{eq-thermal-damping-sum}. 
We believe that the approximations performed in \cite[(3.8)-(3.13)]{Prosperetti-JFM1991}, as well as \cite[(4.2), (4.3)]{Prosperetti-JFM1991}, require further justification due to the subtleties of  two variables, $D=\chi/(\om R_*^2)$ and time $t$, being taken  to infinity.
(Note that a factor $D$ is missing on the right hand side of \cite[(3.11)]{Prosperetti-JFM1991})
For example, the approximation to the delta function displayed above \cite[(3.12)]{Prosperetti-JFM1991} is derived from \cite[(3.11)]{Prosperetti-JFM1991} by fixing a finite time $t$ and then sending $D\to\infty$.
As for the expansion \cite[(4.2), (4.3)]{Prosperetti-JFM1991}, the uniform convergence of the series for all $t>0$ needs to be justified.
Without further justification, we believe it is possible that the approximated equation \cite[(3.13)]{Prosperetti-JFM1991} and \cite[(4.9)]{Prosperetti-JFM1991} may only valid up to any fixed finite time, and thus  may not give the actually exponential time-decay rate as $t\to\infty$.
In contrast, our lower bound, \eqref{eq-thermal-damping-sum}, for $\be_{\rm thermal}$ establishes a decay rate of at least of order $\exp(-\be_{\rm thermal} t)$, which is valid for all time $t>0$.

In the nearly adiabatic case ($\chi$ small) on the other hand, Prosperetti presented numerical results of two approximations: implicit nearly adiabatic approximation \cite[(5.8)]{Prosperetti-JFM1991} and
explicit nearly adiabatic approximation \cite[(5.10)]{Prosperetti-JFM1991} using cubic splines with iteration until convergence \cite[Section 5]{Prosperetti-JFM1991}.
Our $\be_{\rm adiabatic}$ in \eqref{eq-thermal-damping-sum} is the first analytic and quantitative estimate of $\be_{\rm thermal}$ in the nearly adiabatic case.

 \subsection{Decay rate to periodic states}\label{sec-compare-prosperetti-per}

When a small-amplitude time-periodic sound field is exerted,
 it follows from Theorem \ref{cor-period-exituniq} that the thermal damping rate $\beta_{\rm thermal, per}$ (setting $\mu_l=0$) for periodically forced linearized model is $\beta_{\rm thermal, per} =  \beta_{\rm thermal} + o_A(1)$ as $A\to0$.
Thus, from \eqref{eq-thermal-damping-sum}
\EQ{\label{eq-thermal-damping-sum-per}
 \beta_{\rm thermal, per} >
 \begin{cases} 
 \dfrac{4\vartheta(\gamma) p_*}{90\rho_l\chi} = : \beta_{\rm isothermal, per} , \quad &\textrm{for $\chi$ large (isothermal)},\\
 \bke{1 - \sqrt{\dfrac{\vartheta(\gamma)}{\frac{p_{\infty,*}R_*}{2p_{*,\infty}R_*+6\si} + \vartheta(\gamma)}}} \dfrac{\pi^2\chi}{R_*^2} = : \beta_{\rm adiabatic, per},\quad &\textrm{for $\chi$ small (adiabatic)}.
 \end{cases}
 }

We now compare \eqref{eq-thermal-damping-sum-per} with the results for the equivalent linearized forced problem \cite[(3.7)]{Prosperetti-JFM1991} by Prosperetti. 

Upon the approximation \cite[(3.20)]{Prosperetti-JFM1991} for $t$ sufficiently large from Laplace transforms and the assumption $X''' = -X'$, $X=(R/R_*)-1$, 
Prosperetti derived the approximated thermal dissipation \cite[(3.27)]{Prosperetti-JFM1991}
\[
\be_{\rm th, per, P91} 
= \frac{p_*}{2\rho_l\om R_*^2} \Im \wt F\bke{\frac{i}D},\ \text{ where $\wt F$ is defined in \cite[(3.5)]{Prosperetti-JFM1991}.}
\]
Prosperetti argued  that the approximated thermal dissipation rate $\be_{\rm th, per, P91}$ has the asymptotic limits \cite[(3.30), (3.31)]{Prosperetti-JFM1991} for the nearly isothermal ($\chi$ large) and the nearly adiabatic ($\chi$ small) cases, respectively, in terms of the variable $\eta = R_*(2\om/\chi)^{1/2}$. 

For isothermal case, \cite[(3.30)]{Prosperetti-JFM1991} yields 
\EQ{\label{eq-thermal-damping-per-isoth-P91}
\be_{\rm th, per, P91} 
\sim \frac{p_*}{2\rho_l\om R_*^2}\, \frac{\vartheta(\ga)}{10}\, \eta^2
= \frac{\vartheta(\ga)p_*}{10\rho_l\chi}
=: \be_{\rm isothermal, per, P91} \quad\text{ for large $\chi$ (isothermal),}
}
which is exactly the same as $\be_{\rm th, P91}$ in \eqref{eq-thermal-damping-P91} for the unforced one.
The approximated damping rate $\be_{\rm isothermal, per, P91}$ satisfies our lower bound estimate \eqref{eq-thermal-damping-sum-per} and is $9/4$-times greater than our lower bound $\beta_{\rm isothermal, per}$.

For adiabatic case, \cite[(3.27), (3.31)]{Prosperetti-JFM1991} yields
\EQ{\label{eq-thermal-damping-per-adiab-P91}
\be_{\rm th, per, P91} 
\sim \frac{p_*}{2\rho_l\om R_*^2}\, 9\ga\, \frac{\ga-1}\eta
= \frac{9\ga(\ga-1)p_*}{2^{3/2}\rho_l\om^{3/2}R_*^3}\, \chi^{1/2}
=: \be_{\rm adiabatic, per, P91}\quad\text{ for small $\chi$ (adiabatic).}
}
In terms of the order in the thermal diffusivity $\chi$, $\be_{\rm adiabatic, per, P91} = O(\chi^{1/2})$ in \eqref{eq-thermal-damping-per-adiab-P91} is larger than $\be_{\rm thermal, per} = O(\chi)$ in \eqref{eq-thermal-damping-sum-per}.
Thus, $\be_{\rm adiabatic, per, P91}$ satisfies our lower bound \eqref{eq-thermal-damping-sum-per}.

\section{Improved estimate of the linear exponential decay rate $\be$}\label{beta-est}

In this appendix, we investigate the location of the roots of the meromorphic function $Q(\tau)$ defined in \eqref{eq-Q-def}.
The following lemma is an improvement of \cite[Lemma E.1]{LW-vbas2022}.
\begin{lemma}\label{lem-negative-upper-bound}
There exists a negative upper bound for the real parts of the roots of the meromorphic function $Q(\tau)$ in \eqref{eq-Q-def}.
Specifically, there exists $\be>0$ such that $\xi<-\be$ for all roots $\tau = \xi + i\eta$ of $Q(\tau)$.
The constant $\be$ can be chosen as
\EQ{\label{eq-be-sharp-def-appendix}
\be &= \min\Bigg\{ \bke{1 - \sqrt{\dfrac{\vartheta(\ga)}{ \frac{p_{\infty,*}R_*}{2p_{\infty,*}R_*+6\si} + \vartheta(\ga) }}}\pi^2\bar\ka,\  \sqrt{ \e \frac{2p_*}{\rho_lR_*^2} },\\
&\qquad \frac{2\mu_l}{\rho_lR_*^2} + \mathbbm{1}_{\De\le0}\, \frac{(1-\ve)^2\vartheta(\ga)p_*}{\pi^4\bar\ka \rho_lR_*^2} \bke{\frac{4\pi^4}{90} + (1-\ve)^{3/4} O\bke{\bke{\frac1{\pi^2\bar\ka}\sqrt{ \frac{2p_*}{\rho_lR_*^2 }}}^{3/2}}} - \mathbbm{1}_{\De>0}\, \frac{\sqrt{\De}}{2\rho_lR_*} \Bigg\},
}
where $\vartheta(\ga)=1-\ga^{-1}$, $\ve\in(0,1)$ is arbitrary, and $\De:= \bke{ \frac{4\mu_l}{R_*} }^2 - 8\rho_lp_*$.
\end{lemma}

\cite[Lemma E.1]{LW-vbas2022} is a special case Lemma \ref{lem-negative-upper-bound} for $\ve=1/2$.

\begin{proof}[Proof of Lemma \ref{lem-negative-upper-bound}]
We first simplify $Q(\tau)$ in \eqref{eq-Q-def}.
Since
\[
\sum_{j=1}^\infty \frac{\tau}{j^2\bke{\pi^2\bar\ka j^2 + \tau}} = \sum_{j=1}^\infty \bke{\frac1{j^2} - \frac{\pi^2\bar\ka}{\pi^2\bar\ka j^2 + \tau} }
= \frac{\pi^2}6 - \sum_{j=1}^\infty \frac{\pi^2\bar\ka}{\pi^2\bar\ka j^2 + \tau},
\]
\EQ{\label{eq-Q-simplified}
Q(\tau) = \frac1{\Rg T_\infty} \bke{\frac{4\pi}{3\ga} + \frac{8(\ga-1)}{\pi\ga} \sum_{j=1}^\infty \frac{\pi^2\bar\ka}{\pi^2\bar\ka j^2 + \tau} } \bke{\rho_lR_*\tau^2 + \frac{4\mu_l}{R_*}\, \tau - \frac{2\si}{R_*^2}} + 4\pi\,\frac{\rho_*}{R_*}.
}
Let $\tau = \xi + i\eta$ be a root of $Q(\tau)$, i.e., $Q(\tau)=0$. 

Plugging $\tau = \xi + i\eta$, $\xi\in\R,\eta\in\R$, into \eqref{eq-Q-simplified}, we have
\[
Q(\xi + i\eta) = \frac1{\Rg T_\infty} \bke{\Xi_1 + iH_1} \bke{\Xi_2 + iH_2} + 4\pi\,\frac{\rho_*}{R_*},
\]
where
\EQ{\label{eq-Xi-Eta}
\Xi_1 &= \frac{4\pi}{3\ga} + \frac{8(\ga-1)}{\pi\ga}\sum_{j=1}^\infty \frac{\pi^2\bar\ka\bke{\pi^2\bar\ka j^2 + \xi}}{\bke{\pi^2\bar\ka j^2 + \xi}^2 + \eta^2},\\
H_1 &=  - \frac{8(\ga-1)}{\pi\ga} \sum_{j=1}^\infty \frac{\pi^2\bar\ka\eta}{\bke{\pi^2\bar\ka j^2 + \xi}^2 + \eta^2},\\
\Xi_2 &= \rho_lR_*\bke{\xi^2 - \eta^2} + \frac{4\mu_l}{R_*}\, \xi - \frac{2\si}{R_*^2},\\
H_2 &= \rho_lR_*(2\xi\eta) + \frac{4\mu_l}{R_*}\, \eta. 
}
Setting real and imaginary parts of $Q$ equal to zero, we obtain
\EQ{\label{eq-real-imaginary-Q}
\text{real part: }& \frac1{\Rg T_\infty} \bke{\Xi_1\Xi_2 - H_1H_2} + 4\pi\, \frac{\rho_*}{R_*} = 0,\\
\text{imaginary part: }& \frac1{\Rg T_\infty} \bke{\Xi_1H_2 + H_1\Xi_2} = 0.
}
The real part in \eqref{eq-real-imaginary-Q} reads
\EQ{\label{eq-real-part-1}
0 &= \frac1{\Rg T_\infty} \bigg[ \bke{ \frac{4\pi}{3\ga} + \frac{8(\ga-1)}{\pi\ga}\sum_{j=1}^\infty \frac{\pi^2\bar\ka\bke{\pi^2\bar\ka j^2 + \xi}}{\bke{\pi^2\bar\ka j^2 + \xi}^2 + \eta^2} } \bke{\rho_lR_*\bke{\xi^2-\eta^2} + \frac{4\mu_l}{R_*}\, \xi - \frac{2\si}{R_*^2} }\\
&\qquad\qquad\qquad\qquad\qquad + \frac{8(\ga-1)}{\pi\ga} \sum_{j=1}^\infty \frac{\pi^2\bar\ka\eta^2}{\bke{\pi^2\bar\ka j^2 + \xi}^2 + \eta^2 } \bke{\rho_lR_*(2\xi) + \frac{4\mu_l}{R_*}}\bigg] + 4\pi\, \frac{\rho_*}{R_*}\\
&= \frac1{\Rg T_\infty} \bigg[ \bke{ \frac{4\pi}{3\ga} + \frac{8(\ga-1)}{\pi\ga}\sum_{j=1}^\infty \frac{\pi^4\bar\ka^2j^2}{\bke{\pi^2\bar\ka j^2 + \xi}^2 + \eta^2} } \bke{\rho_lR_*\bke{\xi^2-\eta^2} + \frac{4\mu_l}{R_*}\, \xi - \frac{2\si}{R_*^2} }\\
&\qquad\qquad\quad + \frac{8(\ga-1)}{\pi\ga} \sum_{j=1}^\infty \frac{\pi^2\bar\ka}{\bke{\pi^2\bar\ka j^2 + \xi}^2 + \eta^2 } \bke{\rho_lR_*\xi(\xi^2+\eta^2) + \frac{4\mu_l}{R_*}(\xi^2+\eta^2) - \frac{2\si}{R_*^2}\, \xi }\bigg] + 4\pi\, \frac{\rho_*}{R_*}.
}
When $\eta\neq0$, the imaginary part in \eqref{eq-real-imaginary-Q} reads
\EQ{\label{eq-imaginary-part-1}
0 &= \bke{\frac{4\pi}{3\ga} + \frac{8(\ga-1)}{\pi\ga} \sum_{j=1}^\infty \frac{\pi^2\bar\ka\bke{\pi^2\bar\ka j^2 + \xi}}{\bke{\pi^2\bar\ka j^2 + \xi}^2 + \eta^2}} \bke{\rho_lR_*(2\xi) + \frac{4\mu_l}{R_*} }\\
&\qquad\qquad\qquad\qquad\qquad - \frac{8(\ga-1)}{\pi\ga} \sum_{j=1}^\infty \frac{\pi^2\bar\ka}{\bke{\pi^2\bar\ka j^2 + \xi}^2 + \eta^2 } \bke{\rho_lR_*(\xi^2 - \eta^2) + \frac{4\mu_l}{R_*}\, \xi - \frac{2\si}{R_*^2} }\\
&= \bke{\frac{4\pi}{3\ga} + \frac{8(\ga-1)}{\pi\ga} \sum_{j=1}^\infty \frac{\pi^4\bar\ka^2 j^2}{\bke{\pi^2\bar\ka j^2 + \xi}^2 + \eta^2}} \bke{\rho_lR_*(2\xi) + \frac{4\mu_l}{R_*} }\\
&\qquad\qquad\qquad\qquad\qquad + \frac{8(\ga-1)}{\pi\ga} \sum_{j=1}^\infty \frac{\pi^2\bar\ka}{\bke{\pi^2\bar\ka j^2 + \xi}^2 + \eta^2 } \bke{\rho_lR_*(\xi^2 + \eta^2) + \frac{2\si}{R_*^2} }.
}
For $\eta\neq0$, the equation \eqref{eq-imaginary-part-1} implies 
\EQ{\label{eq-imaginary-part-2}
\rho_lR_*(2\xi) + \frac{4\mu_l}{R_*} < 0,
}
which gives
\EQ{\label{eq-xi-first-upper-bound}
\xi < - \frac{2\mu_l}{\rho_lR_*^2} \le 0,\qquad \eta\neq0.
}
The equation \eqref{eq-imaginary-part-1} also implies 
\EQ{\label{eq-imaginary-part-3}
\frac{8(\ga-1)}{\pi\ga} \sum_{j=1}^\infty \frac{\pi^2\bar\ka}{\bke{\pi^2\bar\ka j^2 + \xi}^2 + \eta^2 }
= - \bke{\frac{4\pi}{3\ga} + \frac{8(\ga-1)}{\pi\ga} \sum_{j=1}^\infty \frac{\pi^4\bar\ka^2 j^2}{\bke{\pi^2\bar\ka j^2 + \xi}^2 + \eta^2}} \dfrac{\rho_lR_*(2\xi) + \frac{4\mu_l}{R_*}}{\rho_lR_*(\xi^2+\eta^2) + \frac{2\si}{R_*^2}}.
}
Plugging \eqref{eq-imaginary-part-3} into the real part \eqref{eq-real-part-1}, we derive
\EQ{\label{eq-real-part-2}
0 &= \frac1{\Rg T_\infty} \bke{ \frac{4\pi}{3\ga} + \frac{8(\ga-1)}{\pi\ga}\sum_{j=1}^\infty \frac{\pi^4\bar\ka^2j^2}{\bke{\pi^2\bar\ka j^2 + \xi}^2 + \eta^2} }
\Bigg[ \rho_lR_*\bke{\xi^2-\eta^2} + \frac{4\mu_l}{R_*}\, \xi - \frac{2\si}{R_*^2} \\
&\quad - \dfrac{\bke{ \rho_lR_*(2\xi) + \frac{4\mu_l}{R_*} } \bke{ \rho_lR_*\xi(\xi^2+\eta^2) + \frac{4\mu_l}{R_*}(\xi^2+\eta^2) - \frac{2\si}{R_*^2}\, \xi } }{\rho_lR_*(\xi^2+\eta^2) + \frac{2\si}{R_*^2}} \Bigg] + 4\pi\, \frac{\rho_*}{R_*}\\
&=: \frac1{\Rg T_\infty} \bke{ \frac{4\pi}{3\ga} + \frac{8(\ga-1)}{\pi\ga}\sum_{j=1}^\infty \frac{\pi^4\bar\ka^2j^2}{\bke{\pi^2\bar\ka j^2 + \xi}^2 + \eta^2} } \La + 4\pi\, \frac{\rho_*}{R_*},
}
where $\La$ is the square bracket on the right hand side of the first equation.
A straightforward calculation shows that
\EQN{
&\bke{ \rho_lR_*(\xi^2+\eta^2) + \frac{2\si}{R_*^2} }\La \\
&= -\rho_l^2R_*^2 (\xi^2 + \eta^2)^2 - \frac{4\mu_l}{R_*} \bke{\rho_lR_*(2\xi) + \frac{4\mu_l}{R_*} } (\xi^2 + \eta^2) + \frac{2\si}{R_*^2} \bke{2\rho_lR_*\xi^2 + 2\,\frac{4\mu_l}{R_*}\, \xi - 2\rho_lR_*\eta^2} - \bke{\frac{2\si}{R_*^2}}^2\\
&> -\rho_l^2R_*^2 (\xi^2 + \eta^2)^2 - \frac{2\si}{R_*^2}\, 2\rho_lR_*(\xi^2+\eta^2) - \bke{\frac{2\si}{R_*^2}}^2 = -\bke{\rho_lR_*(\xi^2+\eta^2) + \frac{2\si}{R_*^2}}^2,
}
where we've used \eqref{eq-imaginary-part-2} and so $2\rho_lR_*\xi^2 + 2\,\frac{4\mu_l}{R_*}\, \xi = \xi \bke{\rho_lR_*(2\xi) + 2\,\frac{4\mu_l}{R_*} } > -2\rho_lR_*\xi^2$ in the last inequality.
This implies
\EQ{\label{eq-Lambda-ineq}
\La > -\bke{\rho_lR_*(\xi^2+\eta^2) + \frac{2\si}{R_*^2}}.
}
Using \eqref{eq-Lambda-ineq} in \eqref{eq-real-part-2}, we get
\EQ{\label{eq-real-part-3}
0 &> - \frac1{\Rg T_\infty} \bke{ \frac{4\pi}{3\ga} + \frac{8(\ga-1)}{\pi\ga}\sum_{j=1}^\infty \frac{\pi^4\bar\ka^2j^2}{\bke{\pi^2\bar\ka j^2 + \xi}^2 + \eta^2} } \bke{ \rho_lR_*(\xi^2+\eta^2) + \frac{2\si}{R_*^2} } + 4\pi\, \frac{\rho_*}{R_*}.
}
Suppose 
\EQ{\label{eq-xi-assumption-1}
\xi\ge-\th\pi^2\bar\ka,
} 
where $\th\in(0,1)$ to be chosen. Then $\xi\ge-\th\pi^2\bar\ka j^2$ for all $j=1,2,\ldots$.
We further assume that 
\EQ{\label{eq-xi-assumption-2}
\xi \ge - \sqrt{\frac{2p_{\infty,*}}{\rho_lR_*^2} + \frac{4\si}{\rho_lR_*^3} - \eta^2},
}
provided $\eta^2 \le \frac{2p_{\infty,*}}{\rho_lR_*^2} + \frac{4\si}{\rho_lR_*^3}$, so that $\rho_lR_*(\xi^2+\eta^2) + \frac{2\si}{R_*^2} \le \frac{2p_{\infty,*}}{R_*} + \frac{6\si}{R_*^2}$.
Then \eqref{eq-real-part-3} gives
\EQN{
0 &> - \frac1{\Rg T_\infty} \bke{ \frac{4\pi}{3\ga} + \frac{8(\ga-1)}{\pi\ga}\, \frac1{(1-\th)^2} \sum_{j=1}^\infty \frac1{j^2} } \bke{ \frac{2p_{\infty,*}}{R_*} + \frac{6\si}{R_*^2} } + 4\pi\, \frac{\rho_*}{R_*}.
}
Using $\sum_{j=1}^\infty j^{-2} = \pi^2/6$, one has
\[
3\Rg T_\infty\, \frac{\rho_*}{R_*} < \bkt{\frac1\ga + \bke{1-\frac1\ga} \frac1{(1-\th)^2}} \bke{\frac{2p_{\infty,*}}{R_*} + \frac{6\si}{R_*^2} },
\]
or, equivalently, using $\Rg T_\infty \rho_* = p_{\infty,*} + 2\si/R_*$,
\[
\th > 1 - \sqrt{\dfrac{1-\frac1\ga}{\frac{3p_{\infty,*}R_*+6\si}{2p_{\infty,*}R_*+6\si} - \frac1\ga}}.
\]
We simply choose 
\[
\th = 1 - \sqrt{\dfrac{1-\frac1\ga}{\frac{3p_{\infty,*}R_*+6\si}{2p_{\infty,*}R_*+6\si} - \frac1\ga}} \in (0,1)
\]
to reach a contradiction to \eqref{eq-xi-assumption-1} and \eqref{eq-xi-assumption-2}.
Therefore, we have for $\eta\in\R$ with $\eta\neq0$ and $\eta^2\le \frac{2p_{\infty,*}}{\rho_lR_*^2} + \frac{4\si}{\rho_lR_*^3}$ that
\[
\xi < - \min\bket{ \bke{1 - \sqrt{\dfrac{1-\frac1\ga}{\frac{3p_{\infty,*}R_*+6\si}{2p_{\infty,*}R_*+6\si} - \frac1\ga}}}\pi^2\bar\ka,\, \sqrt{\frac{2p_{\infty,*}}{\rho_lR_*^2} + \frac{4\si}{\rho_lR_*^3} - \eta^2}}.
\]
Combining \eqref{eq-xi-first-upper-bound}, for $\eta\in\R$ with $\eta\neq0$ and $\eta^2\le (1-\e) \bke{ \frac{2p_{\infty,*}}{\rho_lR_*^2} + \frac{4\si}{\rho_lR_*^3} }$ we have
\EQ{\label{eq-xi-upper-bound-eta-small}
\xi < - \max\bket{\frac{2\mu_l}{\rho_lR_*^2},\, \min\bket{ \bke{1 - \sqrt{\dfrac{1-\frac1\ga}{\frac{3p_{\infty,*}R_*+6\si}{2p_{\infty,*}R_*+6\si} - \frac1\ga}}}\pi^2\bar\ka,\, \sqrt{ \e\bke{ \frac{2p_{\infty,*}}{\rho_lR_*^2} + \frac{4\si}{\rho_lR_*^3}} }} }.
}

Now, we consider the case $\eta^2 > (1-\e) \bke{ \frac{2p_{\infty,*}}{\rho_lR_*^2} + \frac{4\si}{\rho_lR_*^3} }$. Since $\eta\neq0$, the imaginary part in \eqref{eq-real-imaginary-Q} gives the identity $\Xi_1 = - \frac{H_1}{H_2}\, \Xi_2$. Using this identity in the real part in \eqref{eq-real-imaginary-Q}, we derive
\[
4\pi\, \frac{\rho_*}{R_*}\, \Rg T_\infty H_2 = H_1\bke{\Xi_2^2 + H_2^2},
\]
which implies
\EQ{\label{eq-Xi-Eta-square}
4\pi\, \frac{\rho_*}{R_*}\, \Rg T_\infty \bke{\rho_lR_*(2\xi) + \frac{4\mu_l}{R_*} } = -\frac{8(\ga-1)}{\pi\ga} \sum_{j=1}^\infty \frac{\pi^2\bar\ka}{\bke{\pi^2\bar\ka j^2 + \xi}^2 + \eta^2}\, \bke{\Xi_2^2 + H_2^2}.
}
To find a positive lower bound for $\Xi_2^2 + H_2^2$,
note that
\[
\Xi_2^2 + H_2^2 = \abs{\rho_lR_*\tau^2 + \frac{4\mu_l}{R_*}\, \tau - \frac{2\si}{R_*^2}}^2 = \rho_l^2R_*^2 \abs{\tau - \tau_+}^2 \abs{\tau - \tau_-}^2,
\]
where 
\[
\tau_{\pm} = \dfrac{-\dfrac{4\mu_l}{R_*} \pm \sqrt{\bke{\dfrac{4\mu_l}{R_*}}^2 + 4\rho_lR_*\, \dfrac{2\si}{R_*^2}}}{2\rho_lR_*}
\]
are on the real axis.
By the triangular inequality $|\tau - \tau_{\pm}| > |\eta|$, and so
\[
\Xi_2^2 + H_2^2 > \rho_l^2 R_*^2 \eta^4.
\]
Therefore, \eqref{eq-Xi-Eta-square} yields
\EQ{\label{eq-eta-large-est}
4\pi\, \frac{\rho_*}{R_*}\, \Rg T_\infty \bke{\rho_lR_*(2\xi) + \frac{4\mu_l}{R_*} } < -\frac{8(\ga-1)}{\pi\ga} \sum_{j=1}^\infty \frac{\pi^2\bar\ka}{\pi^4\bar\ka^2 j^4 + \eta^2}\, \rho_l^2R_*^2\eta^4.
}
Since $\eta^2 > (1-\e) \bke{ \frac{2p_{\infty,*}}{\rho_lR_*^2} + \frac{4\si}{\rho_lR_*^3} }$ and $\frac{a^2}{\pi^4\bar\ka^2 j^4 + a}$ is increasing in $a$ for $a = \eta^2 \ge 0$,
we further derive 
\EQ{\label{eq-Xi-Eta-square-1}
4\pi\, \frac{\rho_*}{R_*}\, \Rg T_\infty& \bke{\rho_lR_*(2\xi) + \frac{4\mu_l}{R_*} } \\
&< -\frac{8(\ga-1)}{\pi\ga} \sum_{j=1}^\infty \frac{\pi^2\bar\ka}{\pi^4\bar\ka^2 j^4 + (1-\e)\bke{\frac{2p_{\infty,*}}{\rho_lR_*^2} + \frac{4\si}{\rho_lR_*^3}}}\, \rho_l^2R_*^2 (1-\e)^2 \bke{\frac{2p_{\infty,*}}{\rho_lR_*^2} + \frac{4\si}{\rho_lR_*^3}}^2\\
&= -\frac{8(\ga-1)}{\pi\ga}\, \frac1{\pi^2\bar\ka} \sum_{j=1}^\infty \frac1{j^4 + B^2}\, \rho_l^2R_*^2 (1-\e)^2 \bke{\frac{2p_{\infty,*}}{\rho_lR_*^2} + \frac{4\si}{\rho_lR_*^3}}^2,
}
where
\[
B:= \frac1{\pi^2\bar\ka}\sqrt{ (1-\e) \bke{\frac{2p_{\infty,*}}{\rho_lR_*^2} + \frac{4\si}{\rho_lR_*^3}}}.
\]
Using 
\EQN{
\sum_{j=1}^\infty \frac1{j^4 + B^2} 
&= \frac{e^{\frac{i\pi}4}\pi\cot\bke{e^{\frac{i\pi}4}\pi\sqrt{B}} + e^{\frac{3i\pi}4}\pi\cot\bke{e^{\frac{3i\pi}4}\pi\sqrt{B}} }{4B^{3/2}} - \frac1{2B^2}\\
&= \dfrac{\frac{2\pi}{\sqrt2} \bke{\frac{\cot\bke{\pi\sqrt{B/2}}\csch^2\bke{\pi\sqrt{B/2}} + \csc^2\bke{\pi\sqrt{B/2}} \coth\bke{\pi\sqrt{B/2}}}{\cot^2\bke{\pi\sqrt{B/2}} + \coth^2\bke{\pi\sqrt{B/2}}} }}{4B^{3/2}} - \frac1{2B^2}\\
& = \dfrac{\frac{2\pi}{\sqrt2} \bke{ \frac1{\pi\sqrt{B/2}} + \frac4{45} \bke{\pi\sqrt{B/2}}^3 + O\bke{\bke{\pi\sqrt{B/2}}^6}}}{4B^{\frac32}} - \frac1{2B^2}\\
& = \frac{\pi^4}{90} + O\bke{B^{3/2}}\ \text{ for }B\ll1,
}
we have from \eqref{eq-Xi-Eta-square-1} that
\EQN{
&4\pi\, \frac{\rho_*}{R_*}\, \Rg T_\infty \bke{\rho_lR_*(2\xi) + \frac{4\mu_l}{R_*} } \\
&< -\frac{8(\ga-1)}{\pi\ga}\, \frac1{\pi^2\bar\ka} \bke{\frac{\pi^4}{90} + O\bke{\bke{\frac1{\pi^2\bar\ka}\sqrt{ (1-\e) \bke{\frac{2p_{\infty,*}}{\rho_lR_*^2} + \frac{4\si}{\rho_lR_*^3}} }}^{3/2}}}  \rho_l^2R_*^2  (1-\e)^2 \bke{\frac{2p_{\infty,*}}{\rho_lR_*^2} + \frac{4\si}{\rho_lR_*^3}}^2.
}
Consequently, we have for $\eta^2 > (1-\e) \bke{ \frac{2p_{\infty,*}}{\rho_lR_*^2} + \frac{4\si}{\rho_lR_*^3} }$ that
\EQ{\label{eq-xi-upper-bound-eta-large}
\xi< -\frac{2\mu_l}{\rho_lR_*^2} - \frac{\Rg T_\infty \rho_*}{\pi^4\bar\ka \rho_lR_*^2} \bke{1-\frac1\ga} \bke{\frac{ 4(1-\e)^2 \pi^4}{90} + O\bke{\bke{\frac{ (1-\e)^{4/3}}{\pi^2\bar\ka}\sqrt{ (1-\e)\bke{\frac{2p_{\infty,*}}{\rho_lR_*^2} + \frac{4\si}{\rho_lR_*^3}} }}^{3/2}}},
}
where $\Rg T_\infty \rho_* = p_{\infty,*} + 2\si/R_*$ is used.

It remains to consider the case $\eta=0$. 
We first show that $\xi<0$.
It is a direct consequence of the asymptotic stability, Theorem \ref{thm-asymp-stab-linear}. Here we give a self-contained proof for the interested readers.
Suppose for the sake of contradiction that $\xi\ge0$ then
\EQN{
Q(\xi) 
&= \frac1{\Rg T_\infty} \bke{\frac{4\pi}{3\ga} + \frac{8(\ga-1)}{\pi\ga} \sum_{j=1}^\infty \frac{\pi^2\bar\ka}{\pi^2\bar\ka j^2 + \xi} } \bke{\rho_lR_*\xi^2 + \frac{4\mu_l}{R_*}\, \xi }\\
&\quad\quad\quad - \frac{2\si}{\Rg T_\infty R_*^2} \bke{\frac{4\pi}{3\ga} + \frac{8(\ga-1)}{\pi\ga} \sum_{j=1}^\infty \frac{\pi^2\bar\ka}{\pi^2\bar\ka j^2 + \xi} } + 4\pi\,\frac{\rho_*}{R_*},
}
where the second line is greater than
\[
 - \frac{2\si}{\Rg T_\infty R_*^2} \bke{\frac{4\pi}{3\ga} + \frac{8(\ga-1)}{\pi\ga} \sum_{j=1}^\infty \frac{\pi^2\bar\ka}{\pi^2\bar\ka j^2} } + 4\pi\,\frac{\rho_*}{R_*}
=  \frac{8\pi\rho_*}{3R_*} + \frac{4\pi p_{\infty,*}}{3\Rg T_\infty R_*} > 0.
\]
Here the identities $\sum_{j=1}^\infty j^{-2} = \frac{\pi^2}6$ and $\Rg T_\infty\rho_* = p_{\infty,*} + \frac{2\si}{R_*}$ are used. 
This yields $Q(\xi)>0$, which contradicts to the fact that $Q(\tau)=0$. Thus, we have $\xi<0$ for $\eta=0$.

Now we search for a negative upper bound for $\xi$ when $\eta=0$.
Suppose 
\EQ{\label{eq-xi-assumption-3}
\xi \ge -\th_0\pi^2\bar\ka,
}
where $0<\th_0<1$ to be chosen. Suppose further that 
\EQ{\label{eq-xi-assumption-4}
\xi > \frac{-\frac{4\mu_l}{R_*} + \sqrt{\De}}{2\rho_lR_*}\quad \text{ if }\De:= \bke{ \frac{4\mu_l}{R_*} }^2 - 4\rho_lR_* \bke{ \frac{2p_{\infty,*}}{R_*} + \frac{4\si}{R_*^2}} >0
}
such that 
\EQ{\label{eq-xi-assumption-4-consequence}
\rho_lR_*\xi^2 + \frac{4\mu_l}{R_*}\, \xi - \frac{2\si}{R_*^2} \ge -\frac{2p_{\infty,*}}{R_*} - \frac{6\si}{R_*^2}.
} 
Note that  the inequality \eqref{eq-xi-assumption-4-consequence} always holds when $\De\le0$.
Then
\EQN{
0 = Q(\xi) &= \frac1{\Rg T_\infty} \bke{\frac{4\pi}{3\ga} + \frac{8(\ga-1)}{\pi\ga} \sum_{j=1}^\infty \frac{\pi^2\bar\ka}{\pi^2\bar\ka j^2 + \xi} } \bke{\rho_lR_*\xi^2 + \frac{4\mu_l}{R_*}\, \xi - \frac{2\si}{R_*^2}} + 4\pi\,\frac{\rho_*}{R_*}\\
&> - \frac1{\Rg T_\infty} \bke{\frac{4\pi}{3\ga} + \frac{8(\ga-1)}{\pi\ga(1-\th_0)} \sum_{j=1}^\infty \frac{\pi^2\bar\ka}{\pi^2\bar\ka j^2} } \bke{ \frac{2p_{\infty,*}}{R_*} + \frac{6\si}{R_*^2}} + 4\pi\,\frac{\rho_*}{R_*}\\
&= - \frac1{\Rg T_\infty} \bke{\frac{4\pi}{3\ga} + \frac{4\pi}3 \bke{1-\frac1\ga} \frac1{1-\th_0} } \bke{ \frac{2p_{\infty,*}}{R_*} +  \frac{6\si}{R_*^2}} + 4\pi\,\frac{\rho_*}{R_*}
}
So
\[
4\pi\Rg T_\infty\, \frac{\rho_*}{R_*} < \bke{\frac{4\pi}{3\ga} + \frac{4\pi}3 \bke{1-\frac1\ga} \frac1{1-\th_0} } \bke{\frac{2p_{\infty,*}}{R_*} + \frac{6\si}{R_*^2} },
\]
or, equivalently,
\[
\th_0 > 1 - \dfrac{1-\dfrac1\ga }{\dfrac{3\Rg T_\infty\, \frac{\rho_*}{R_*}}{\frac{2p_{\infty,*}}{R_*} + \frac{6\si}{R_*^2}} - \dfrac1{\ga} } 
= 1 - \dfrac{1-\dfrac1\ga }{ \frac{3p_{\infty,*}R_*+6\si}{2p_{\infty,*}R_*+6\si} - \dfrac1{\ga} },
\]
where $\Rg T_\infty \rho_* = p_{\infty,*} + \frac{2\si}{R_*}$ has been used in the last equation.
We then choose 
\[
\th_0 = 1 - \dfrac{1-\dfrac1\ga }{ \dfrac{3p_{\infty,*}R_*+6\si}{2p_{\infty,*}R_*+6\si} - \dfrac1{\ga} } \in (0,1)
\]
to reach a contradiction to \eqref{eq-xi-assumption-3} and \eqref{eq-xi-assumption-4}. 
Hence we derive for $\eta=0$
\EQ{\label{eq-xi-upper-bound-eta-zero}
\xi < 
\begin{cases}
- \min\bket{ \bke{ 1 - \dfrac{1-\dfrac1\ga }{ \dfrac{3p_{\infty,*}R_*+6\si}{2p_{\infty,*}R_*+6\si} - \dfrac1{\ga} } }\pi^2\bar\ka,\, \dfrac{ \dfrac{4\mu_l}{R_*} - \sqrt{\De}}{2\rho_lR_*} }&\ \text{ if }\De>0,\\
- \bke{ 1 - \dfrac{1-\dfrac1\ga }{ \dfrac{3p_{\infty,*}R_*+6\si}{2p_{\infty,*}R_*+6\si} - \dfrac1{\ga} } }\pi^2\bar\ka&\ \text{ if }\De\le0.
\end{cases}
}
Combining the upper bounds \eqref{eq-xi-upper-bound-eta-small}, \eqref{eq-xi-upper-bound-eta-large}, and \eqref{eq-xi-upper-bound-eta-zero} for different cases of $\eta$ and using the identity $p_* = \Rg T_\infty\rho_* = p_{\infty,*} + \frac{2\si}{R_*}$, the lemma follows.
\end{proof}

Let $\mathcal L_-$ be defined in \eqref{eq-def-L-}.
To show the operator $-\mathcal L_-$ is sectorial, we derive the following finer estimate of the location of the roots of $Q(\tau)$.

\begin{lemma}\label{lem-sectorial-bound}
There exists $\phi\in(\pi/2,\pi)$ such that the roots of the meromorphic function $Q(\tau)$ defined in \eqref{eq-Q-def} lie in the sector
\[
S_\phi = \{\tau\in\CC : \phi \le |{\rm arg}(\tau)|\le \pi\}.
\]
\end{lemma}
\begin{proof}
The proof is based on the calculations in the proof of Lemma \ref{lem-negative-upper-bound}.
For simplicity, we choose $\ve=1/2$. 
Let $\tau = \xi + i\eta$ be a root of $Q(\tau)$.
It follows from \eqref{eq-xi-upper-bound-eta-small} and \eqref{eq-xi-upper-bound-eta-zero} that there exists $C>0$ such that $\xi<-C$ for $\eta^2\le \frac{p_{\infty,*}}{\rho_lR_*^2} + \frac{\si}{\rho_lR_*^3}$.
For $\eta^2> \frac{p_{\infty,*}}{\rho_lR_*^2} + \frac{\si}{\rho_lR_*^3}$, \eqref{eq-eta-large-est} implies that
\[
4\pi\, \frac{\rho_*}{R_*}\, \Rg T_\infty \bke{\rho_lR_*(2\xi) + \frac{4\mu_l}{R_*} } < -\frac{8(\ga-1)}{\pi\ga} \sum_{j=1}^\infty \frac{\pi^2\bar\ka B }{\pi^4\bar\ka^2 j^4 + B }\, \rho_l^2R_*^2\eta^2,\quad
B = \frac{p_{\infty,*}}{\rho_lR_*^2} + \frac{\si}{\rho_lR_*^3}.
\]
Therefore, $\tau = \xi + i\eta$ lies to the left  of a sideways parabola for large $|\eta|$. 
This proves the lemma. 
\end{proof}






\begin{thebibliography}{10}

\bibitem{Leighton2011review}
M.~A. Ainslie and T.~G. Leighton.
\newblock Review of scattering and extinction cross-sections, damping factors,
  and resonance frequencies of a spherical gas bubble.
\newblock {\em The Journal of the Acoustical Society of America},
  130(5):3184--3208, 2011.

\bibitem{Apfel-Ultrasonics1981}
R.~Apfel.
\newblock Acoustic cavitation inception.
\newblock {\em Ultrasonics}, 22(4):167--173, 1984.

\bibitem{barber1991observation}
B.~P. Barber and S.~J. Putterman.
\newblock Observation of synchronous picosecond sonoluminescence.
\newblock {\em Nature}, 352(6333):318--320, 1991.

\bibitem{barber1992light}
B.~P. Barber and S.~J. Putterman.
\newblock Light scattering measurements of the repetitive supersonic implosion
  of a sonoluminescing bubble.
\newblock {\em Phys. Rev. Lett.}, 69(26):3839, 1992.

\bibitem{barber1994sensitivity}
B.~P. Barber, C.~Wu, R.~L{\"o}fstedt, P.~H. Roberts, and S.~J. Putterman.
\newblock Sensitivity of sonoluminescence to experimental parameters.
\newblock {\em Phys. Rev. Lett.}, 72(9):1380, 1994.

\bibitem{BV-SIMA2000}
Z.~Biro and J.~J.~L. Velazquez.
\newblock Analysis of a free boundary problem arising in bubble dynamics.
\newblock {\em SIAM J. Math. Anal.}, 32(1):142--171, 2000.

\bibitem{Brenner-sonoluminescence}
M.~P. Brenner, S.~Hilgenfeldt, and D.~Lohse.
\newblock Single-bubble sonoluminescence.
\newblock {\em Reviews of modern physics}, 74(2):425, 2002.

\bibitem{Fanelli-Prosperetti-Reali-Acustica(47)1981}
M.~Fanelli, A.~Prosperetti, and M.~Reali.
\newblock Radial oscillations of gas-vapor bubbles in liquids part i:
  Mathematical formulation.
\newblock {\em Acta Acustica united with Acustica}, 47(4):253--265, 1981.

\bibitem{Fanelli-Prosperetti-Reali-Acustica(49)1981}
M.~Fanelli, A.~Prosperetti, and M.~Reali.
\newblock Radial oscillations of gas-vapour bubbles in liquids part ii:
  Numerical examples.
\newblock {\em Acta Acustica united with Acustica}, 49(2):98--109, 1981.

\bibitem{Flynn-JAcoustSocAm1975}
H.~G. Flynn.
\newblock Cavitation dynamics. i. a mathematical formulation.
\newblock {\em The Journal of the Acoustical Society of America},
  57(6):1379--1396, 1975.

\bibitem{Hegedus-thesis2018}
F.~Heged{\H{u}}s.
\newblock {\em Spherical bubble dynamics and cavitating vortex shedding}.
\newblock PhD thesis, Budapest University of Technology and Economics
  (Hungary), 2018.

\bibitem{Henry-book1981}
D.~Henry.
\newblock {\em Geometric theory of semilinear parabolic equations}, volume 840
  of {\em Lecture Notes in Mathematics}.
\newblock Springer-Verlag, Berlin-New York, 1981.

\bibitem{KP-JASA1989}
V.~Kamath and A.~Prosperetti.
\newblock Numerical integration methods in gas-bubble dynamics.
\newblock {\em The Journal of the Acoustical Society of America},
  85(4):1538--1548, 1989.

\bibitem{Keller-Miksis-JAcoustSocAm1980}
J.~B. Keller and M.~Miksis.
\newblock Bubble oscillations of large amplitude.
\newblock {\em The Journal of the Acoustical Society of America},
  68(2):628--633, 1980.

\bibitem{LW-vbas2022}
C.-C. Lai and M.~I. Weinstein.
\newblock Free boundary problem for a gas bubble in a liquid, and exponential
  stability of the manifold of spherically symmetric equilibria.
\newblock {\em arXiv preprint arXiv:2207.04079}.

\bibitem{Lastman-Wentzell-JAcoustSocAm1981}
G.~Lastman and R.~Wentzell.
\newblock Comparison of five models of spherical bubble response in an inviscid
  compressible liquid.
\newblock {\em The Journal of the Acoustical Society of America},
  69(3):638--642, 1981.

\bibitem{Lastman-Wentzell-JAcoustSocAm1982}
G.~Lastman and R.~Wentzell.
\newblock On two equations of radial motion of a spherical gas-filled bubble in
  a compressible liquid.
\newblock {\em The Journal of the Acoustical Society of America},
  71(4):835--838, 1982.

\bibitem{Lauterborn-JAcoustSocAm1976}
W.~Lauterborn.
\newblock Numerical investigation of nonlinear oscillations of gas bubbles in
  liquids.
\newblock {\em The Journal of the Acoustical Society of America},
  59(2):283--293, 1976.

\bibitem{Pazy-book1983}
A.~Pazy.
\newblock {\em Semigroups of linear operators and applications to partial
  differential equations}, volume~44 of {\em Applied Mathematical Sciences}.
\newblock Springer-Verlag, New York, 1983.

\bibitem{Plesset-Prosperetti-1977}
M.~S. Plesset and A.~Prosperetti.
\newblock Bubble dynamics and cavitation.
\newblock {\em Annual review of fluid mechanics}, 9(1):145--185, 1977.

\bibitem{Prosperetti-JAcoustSocAm1974}
A.~Prosperetti.
\newblock Nonlinear oscillations of gas bubbles in liquids: steady-state
  solutions.
\newblock {\em The Journal of the Acoustical Society of America},
  56(3):878--885, 1974.

\bibitem{Prosperetti-JAcoustSocAm1975}
A.~Prosperetti.
\newblock Nonlinear oscillations of gas bubbles in liquids. {T}ransient
  solutions and the connection between subharmonic signal and cavitation.
\newblock {\em The Journal of the Acoustical Society of America},
  57(4):810--821, 1975.

\bibitem{Prosperetti-JAcoustSocAm1977}
A.~Prosperetti.
\newblock Thermal effects and damping mechanisms in the forced radial
  oscillations of gas bubbles in liquids.
\newblock {\em The Journal of the Acoustical Society of America}, 61(1):17--27,
  1977.

\bibitem{Prosperetti-ApplSciRes1982}
A.~Prosperetti.
\newblock Bubble dynamics: a review and some recent results.
\newblock {\em Appl. Sci. Res.}, 38:145--164, 1982.

\bibitem{Prosperetti-Ultrasonics1984a}
A.~Prosperetti.
\newblock Bubble phenomena in sound fields: part one.
\newblock {\em Ultrasonics}, 22(2):69--77, 1984.

\bibitem{Prosperetti-Ultrasonics1984b}
A.~Prosperetti.
\newblock Bubble phenomena in sound fields: part two.
\newblock {\em Ultrasonics}, 22(3):115--124, 1984.

\bibitem{Prosperetti-JFM1991}
A.~Prosperetti.
\newblock The thermal behaviour of oscillating gas bubbles.
\newblock {\em J. Fluid Mech.}, 222:587--616, 1991.

\bibitem{SW-SIMA2011}
A.~M. Shapiro and M.~I. Weinstein.
\newblock Radiative decay of bubble oscillations in a compressible fluid.
\newblock {\em SIAM J. Math. Anal.}, 43(2):828--876, 2011.

\bibitem{SBB:87}
P.~Smereka, B.~Birnir, and S.~Banerjee.
\newblock Regular and chaotic bubble oscillations in periodically driven
  pressure fields.
\newblock {\em Phys. Fluids}, 30(11):3342--3350, 1987.

\bibitem{ZP-JFM2020}
G.~Zhou and A.~Prosperetti.
\newblock Modelling the thermal behaviour of gas bubbles.
\newblock {\em J. Fluid Mech.}, 901:R3, 15, 2020.

\end{thebibliography}


\end{document}